\documentclass[a4paper]{amsart}
\usepackage{amssymb, enumitem}
\usepackage[all]{xy}
\usepackage{hyperref, aliascnt}

\newtheorem{lma}{Lemma}[section]

\newaliascnt{thmCt}{lma}
\newtheorem{thm}[thmCt]{Theorem}
\aliascntresetthe{thmCt}

\newaliascnt{corCt}{lma}
\newtheorem{cor}[corCt]{Corollary}
\aliascntresetthe{corCt}

\newaliascnt{prpCt}{lma}
\newtheorem{prp}[prpCt]{Proposition}
\aliascntresetthe{prpCt}

\theoremstyle{definition}

\newaliascnt{pgrCt}{lma}
\newtheorem{pgr}[pgrCt]{}
\aliascntresetthe{pgrCt}

\newaliascnt{dfnCt}{lma}
\newtheorem{dfn}[dfnCt]{Definition}
\aliascntresetthe{dfnCt}

\newaliascnt{rmkCt}{lma}
\newtheorem{rmk}[rmkCt]{Remark}
\aliascntresetthe{rmkCt}

\newaliascnt{exaCt}{lma}
\newtheorem{exa}[exaCt]{Example}
\aliascntresetthe{exaCt}

\newaliascnt{qstCt}{lma}

\aliascntresetthe{qstCt}

\newaliascnt{pbmCt}{lma}
\newtheorem{pbm}[pbmCt]{Problem}
\aliascntresetthe{pbmCt}

\newaliascnt{cnjCt}{lma}
\newtheorem{cnj}[cnjCt]{Conjecture}
\aliascntresetthe{cnjCt}

\def\today{\number\day\space\ifcase\month\or   January\or February\or
   March\or April\or May\or June\or   July\or August\or September\or
   October\or November\or December\fi\   \number\year}

\newcommand{\ZZ}{{\mathbb{Z}}}
\newcommand{\NN}{{\mathbb{N}}}
\newcommand{\CC}{{\mathbb{C}}}
\newcommand{\RR}{{\mathbb{R}}}
\newcommand{\TT}{{\mathbb{T}}}
\newcommand{\Bdd}{{\mathcal{B}}}
\newcommand{\Cpct}{{\mathcal{K}}}
\newcommand{\spec}{{\mathrm{sp}}}

\newcommand{\supp}{{\mathrm{supp}}}
\newcommand{\andSep}{\,\,\,\text{ and }\,\,\,}
\newcommand{\red}{{\mathrm{red}}}

\newcommand{\OP}{\mathcal{O}}
\newcommand{\ca}{$C^*$-algebra}

\DeclareMathOperator{\Proj}{Proj}
\DeclareMathOperator{\tr}{tr}

\newcounter{introCounter}

\newaliascnt{thmIntroCt}{introCounter}
\newtheorem{thmIntro}[thmIntroCt]{Theorem}
\aliascntresetthe{thmIntroCt}

\newtheorem{prpIntro}[introCounter]{Proposition}
\newtheorem{corIntro}[introCounter]{Corollary}

\title{Diffuse traces and Haar unitaries}
\date{\today}

\author{Hannes Thiel}
\address{Hannes Thiel
Mathematisches Institut, Fachbereich Mathematik und Informatik der
Universit\"at M\"unster, Einsteinstrasse 62, 48149 M\"unster, Germany.}
\email{hannes.thiel@posteo.de}
\urladdr{www.hannesthiel.org}

\thanks{The author was partially supported by the Deutsche Forschungsgemeinschaft (DFG, German Research Foundation) under Germany's Excellence Strategy EXC 2044-390685587 (Mathematics M\"{u}nster: Dynamics-Geometry-Structure).
}
\subjclass[2010]%
{Primary
46L05, % General theory of C*-algebras
46L51; % Noncommutative measure and integration
Secondary
46L09, % Free products of $C^*$-algebras
46L30. % States of selfadjoint operator algebras
}
\date{\today}

\begin{document}

%==========================================================================================
\begin{abstract}
We show that a tracial state on a unital \ca{} admits a Haar unitary if and only if it is diffuse, if and only if it does not dominate a tracial functional that factors through a finite-dimensional quotient.
It follows that a unital \ca{} has no finite-dimensional representations if and only if each of its tracial states admits a Haar unitary.

More generally, we study when nontracial states admit Haar unitaries.
In particular, we show that every state on a unital, simple, infinite-dimensional \ca{} admits a Haar unitary.

We obtain applications to the structure of reduced free products.
Notably, the tracial reduced free product of simple \ca{s} is always a simple \ca{} of stable rank one.
\end{abstract}

\maketitle

%==========================================================================================
%==========================================================================================
\section{Introduction}

%==========================================================================================
Let $A$ be a unital \ca{} with a tracial state $\tau\colon A\to\CC$.
A unitary $u\in A$ is called a \emph{Haar unitary} (with respect to $\tau$) if $\tau(u^k)=0$ for $k\in\ZZ\setminus\{0\}$.
Equivalently, the sub-\ca{} $C^*(u)$ of $A$ generated by $u$ is isomorphic to $C(\TT)$ and the restriction of $\tau$ to $C^*(u)$ corresponds to the normalized Lebesgue measure on $\TT$.
%Equivalently, $u$ has full spectrum $\TT$ on which $\tau$ induced the normalized Lebesgue measure.

Haar unitaries play important roles in various constructions and structure results in operator algebras, starting with the fundamental fact that they are natural generators of diffuse abelian *-subalgebras in separable, tracial von Neumann algebras.
Haar unitaries also naturally arise in reduced group \ca{s}, in group von Neumann algebras, and as generators of the irrational rotation \ca{s}. 

Popa's solution, \cite{Pop87CommutantModuloCompacts}, to the commutant modulo compact operators problem for general von Neumann algebras relies on the construction of suitable Haar unitaries in $\mathrm{II}_1$ factors.
This construction was later refined, \cite{Pop95FreeIndepSeq}, to prove the existence of free independent sequences of Haar unitaries for irreducible subfactors of $\mathrm{II}_1$ factors.
Haar unitaries also play a prominent role in Voiculescu's free probability theory, \cite{Dyk93FreeProdHyperfinite}, \cite{DykHaa01InvSubspCircularOps},
in particular as the unitaries in the polar decompositions of R-diagonal elements, \cite{NicSpe97RdiagonalPairs}, \cite[Section~15]{NicSpe06LecturesFreeProb}.

The main result of this paper characterizes when $\tau$ admits a Haar unitary:

%==========================================================================================
\begin{thmIntro}[{\ref{prp:main}}]
\label{thmIntroMain}
The following are equivalent:
\begin{enumerate}
\item
$\tau$ is diffuse (the normal extension $\tau\colon A^{**}\to\CC$ vanishes on every minimal projection; see \autoref{pgr:diffuse})
\item
$\tau$ does not dominate a nonzero tracial functional that factors through a finite-dimensional quotient of $A$;
\item 
there exists a unital (maximal) abelian sub-\ca{} $C(X)\subseteq A$ such that $\tau$ induces a diffuse measure on $X$;
\item
there exists a Haar unitary in $A$ (with respect to $\tau$).
\end{enumerate}
\end{thmIntro}

%==========================================================================================
%It follows that in certain situations a tracial state automatically admits a Haar unitary:

%==========================================================================================
\begin{corIntro}[\ref{prp:charNoFDRep}, \ref{prp:traceSimple}]
A unital \ca{} has no finite-dimensional representations if and only if each of its tracial states admits a Haar unitary.

In particular, every tracial state on a unital, simple, infinite-dimensional \ca{} is diffuse and admits a Haar unitary.
\end{corIntro}

%==========================================================================================
In \autoref{sec:groups}, we apply these results to reduced group \ca{s}.
First, we see that a discrete group $G$ is infinite if and only if the canonical tracial state on the reduced group \ca{} $C^*_\red(G)$ admits a Haar unitary;
see \autoref{prp:charGpInfinite}.
Thus, if $G$ is an infinite, locally finite group, then $C^*_\red(G)$ contains a Haar unitary, while there exists no Haar unitary in $\CC[G]$;
see \autoref{exa:locFinGp}.
We also obtain a characterization of nonamenability:

%==========================================================================================
\begin{prpIntro}[\ref{prp:grpNoFD}]
A discrete group $G$ is nonamenable if and only if every tracial state on $C^*_\red(G)$ admits a Haar unitary.
\end{prpIntro}

%==========================================================================================
In \autoref{sec:freeProd}, we obtain applications to reduced free products.
Given unital \ca{s} $A$ and $B$ with faithful tracial states $\tau_A$ and $\tau_B$, respectively, it is a well-studied problem to determine when the reduced free product of $(A,\tau_A)$ and $(B,\tau_B)$ is simple or has stable rank one;
see \cite{DykHaaRor97SRFreeProd, Dyk99SimplSRFreeProd}.
By \cite[Theorem~2]{Dyk99SimplSRFreeProd}, a sufficient condition is that there is an abelian subalgebra $C(X)\subseteq A$ such that $\tau_A$ induces a diffuse measure on $X$, and that $B\neq\CC$.
\autoref{thmIntroMain} shows that the condition on $A$ is satisfied if and only if $\tau_A$ is diffuse;
see \autoref{prp:freeProdDiffuse}.
As an important special case, we get:

%==========================================================================================
\begin{corIntro}[\ref{prp:freeProdSimpleTracial}]
Let $A$ and $B$ be unital, simple \ca{s} with tracial states $\tau_A$ and $\tau_B$.
Assume that $A\neq\CC$ and $B\neq\CC$.
Then the reduced free product of $(A,\tau_A)$ and $(B,\tau_B)$ is simple, has stable rank one and a unique tracial state.
\end{corIntro}

%==========================================================================================
It follows that the class of unital, simple, stable rank one \ca{s} with unique tracial state is closed under formation of reduced free products;
see \autoref{prp:freeProdSimpleSR1}.
\\

%==========================================================================================
Special cases of \autoref{thmIntroMain} have been shown before.
Under the additional assumption that the \ca{} is abelian, it follows from \cite[Proposition~4.1(i)]{DykHaaRor97SRFreeProd}.
For normal traces on von Neumann algebras, it is also well-known.
In fact, given a normal trace $\tau$ on a diffuse von Neumann algebra $M$, every maximal abelian subalgebra (masa) $D\subseteq M$ is a diffuse sub-von Neumann algebra.
It follows that $\tau|_D$ is a diffuse trace and thus admits a Haar unitary (for instance by \cite[Proposition~4.1(i)]{DykHaaRor97SRFreeProd}).
In particular, every masa in $M$ contains a Haar unitary. % with respect to $\tau$.

We point out that the analogous result does not hold for \ca{s}:
Given a diffuse trace on a unital \ca{} $A$, not every masa of $A$ needs to contain a Haar unitary;
see \autoref{exa:masa}.
This also indicates why the construction of a Haar unitary for a given diffuse trace is rather delicate.

%==========================================================================================
\subsection*{Methods}

%==========================================================================================
The most difficult implication in \autoref{thmIntroMain} is to show that a diffuse tracial state $\tau$ on a unital \ca{} $A$ admits a Haar unitary.
It suffices to construct a positive element in $A$ with spectrum $[0,1]$ on which $\tau$ induces the Lebesgue measure.
To find such element, we first establish a correspondence between positive elements in $A$ and certain maps from $(-\infty,0]$ to the lattice $\OP(A)$ of open projections.
More precisely, a positive element $a$ induces a map $f_a\colon(-\infty,0]\to\OP(A)$, that sends $t\leq 0$ to the support projection of $(a+t)_+$, and one can describe explicitly which maps $(-\infty,0]\to\OP(A)$ arise this way;
see \autoref{prp:elementFromPath}.
Further, a positive element $a$ has spectrum $\spec(a)=[0,1]$ and $\tau$ induces the Lebesgue measure on $\spec(a)$ if and only if the associated map $f_a$ satisfies $\tau(f_a(t))=1+t$ for $t\in[-1,0]$.

To obtain the desired map $f\colon(-\infty,0]\to\OP(A)$, it suffices to construct open projections $p_t$ for dyadic rationals in $[-1,0]$ such that $\tau(p_t)=1+t$ and such that $p_{t'}$ is compactly contained in $p_t$ in the sense of \autoref{dfn:prec}, denoted $p_{t'}\prec p_t$, whenever $t'<t$.
Such projections could easily be obtained by successive interpolation if we could show that for given $p\prec q$ in $\OP(A)$ and $t$ with $\tau(p)<t<\tau(q)$, there exists $r\in\OP(A)$ such that $p\prec r\prec q$ and $\tau(r)=t$.
In \autoref{prp:interpolatingPairs}, we establish an approximate version of this interpolation result, which suffices to construct the desired map $(-\infty,0]\to\OP(A)$.

The crucial assumption in \autoref{prp:interpolatingPairs} is that $\tau$ has `no gaps' in the sense that for every open projection $p$, the set $\{\tau(p') : p'\in\OP(A), p'\leq p\}$ is dense in $[0,\tau(p)]$.
The key observation is that this `no gaps' property holds if (and only if) $\tau$ is \emph{nowhere scattered}, which means that it gives no weight to scattered ideal-quotients;
see \autoref{dfn:NWS}.
This even holds for arbitrary positive functionals, and in \autoref{prp:charNWS} we provide several characterizations for a functional to be nowhere scattered.
The final ingredient is that a diffuse trace is nowhere scattered;
see \autoref{prp:diffuseImpliesNWS}.

%==========================================================================================
\subsection*{Notation}
Given a \ca{} $A$, we use $A_+$ to denote the positive elements in~$A$.
By an ideal in a \ca{} we always mean a closed, two-sided ideal.
Given $a\in A$, we use $\spec(a)$ to denote its spectrum, and we let $\supp(a)$ denote the support projection in $A^{**}$.
Given a Hilbert space $H$, $\Bdd(H)$ denotes the \ca{} of bounded, linear operators on $H$, and $\Cpct(H)$ is the ideal of compact operators.

%==========================================================================================
%==========================================================================================
\section{Open projections}
\label{sec:OP}

%==========================================================================================
Let $A$ be a \ca. 
A projection $p\in A^{**}$ is said to be \emph{open} if there exists an increasing net in $A_+$ that converges to $p$ in the weak*-topology;
a projection $p\in A^{**}$ is said to be \emph{closed} if $1-p$ is open;
see \cite[Definition~II.1]{Ake69StoneWeierstrass}.
Given an open projection $p$, the sub-\ca{} $pA^{**}p\cap A$ of $A$ is hereditary and $p$ is the weak*-limit of any approximate unit of $pA^{**}p\cap A$.
Conversely, given a hereditary sub-\ca{} $B\subseteq A$, there exists a (unique) open projection $p$ such that $B=pA^{**}p\cap A$.
For details we refer to \cite[p.77f]{Ped79CAlgsAutGp}.

We let $\OP(A)$ denote the collection of open projections in $A^{**}$, and we consider it as a subset of the complete lattice $\Proj(A^{**})$ of projections in $A^{**}$.
%We equip $\OP(A)$ with the partial order given by $p\leq q$ if $p=pq$.
%This is just the restriction of the usual order on the family $\Proj(A^{**})$ of projections in $A^{**}$ into a complete lattice.
By \cite[Proposition~II.5]{Ake69StoneWeierstrass}, the infimum of an arbitrary family of closed projections is again closed.
Given any projection $p\in A^{**}$, this allows one to define its \emph{closure} $\overline{p}$ as the smallest closed projection that majorizes $p$.
Dually, $\OP(A)$ is closed under arbitrary suprema in $\Proj(A^{**})$.
Given a family $(p_j)_j$ of open projections, the open projection $\bigvee_j p_j$ corresponds to the hereditary sub-\ca{} of $A$ generated by $A\cap \bigcup_j p_jA^{**}p_j$.
We note that the infimum $\bigwedge_j p_j$ in $\Proj(A^{**})$ is in general strictly larger than the open projection
\[
\bigvee \big\{ p\in\OP(A) : p\leq p_j \text{ for all j} \big\},
\]
which corresponds to the hereditary sub-\ca{} $A \cap \bigcap_j p_jA^{**}p_j$.
Thus, $\OP(A)$ is naturally isomorphic to the lattice of hereditary sub-\ca{s} studied in \cite{AkeBic15HerSubalgLattice}.

%==========================================================================================
\begin{dfn}
\label{dfn:prec}
We define an auxiliary relation $\prec$ on $\OP(A)$ by setting $p\prec q$ for $p,q\in\OP(A)$ if there exists $a\in A_+$ with $p\leq a\leq q$.
\end{dfn}

%==========================================================================================
One can show that $p\prec q$ if and only if $p$ is compactly contained in $q$ in the sense of \cite[Definition~3.6]{OrtRorThi11CuOpenProj}.
Using Theorem~II.5 and Lemma~III.1 in \cite{Ake71GelfandRepr}, it follows that $p\prec q$ if and only if $\overline{p}\leq q$ and there exists $b\in A_+$ such that $\overline{p}\leq b$.
Thus, if $A$ is unital, then $p\prec q$ if and only if $\overline{p}\leq q$.

%==========================================================================================
\begin{lma}
\label{prp:charCompPosProj}
Let $a$ be a positive, \emph{contractive} element, and let $e$ be a projection in a \ca.
Then $e\leq a$ if and only if $e=ea$.
Further, we have $a\leq e$ if and only if $a=ae$.
\end{lma}
\begin{proof}
If $e\leq a\leq 1$, then $e\leq eae\leq e$ and so
\[
[e(1-a)^{1/2}][e(1-a)^{1/2}]^*=e(1-a)e=0.
\]
Hence, $e(1-a)^{1/2}=0$, which implies $e(1-a)=0$, that is, $a=ea$. Conversely, if $e=ea$, then $e=aea\leq a^2$, and so $e=e^{1/2}\leq (a^2)^{1/2}=a$.

Further, if $a\leq e$, then $0\leq (1-e)a(1-e)\leq (1-e)e(1-e)=0$ and so $[a^{1/2}(1-e)]^*[a^{1/2}(1-e)]=0$. 
Hence $a^{1/2}(1-e)=0$, and we get $a(1-e)=0$, that is, $a=ae$;
see \cite[Proposition~II.3.3.2]{Bla06OpAlgs}.
Conversely, if $a=ae$, then $a=eae\leq e^2=e$. 
%Thus, we have $e\prec f$ if and only if there exists $a\in A_+$ such that $a=ea$ and a=af$.
\end{proof}

%==========================================================================================
It follows from \autoref{prp:charCompPosProj} that open projections $p,q$ satisfy $p\prec q$ if and only if there exists a positive element $a$ such that $p=pa$ and $a=aq$.
The next result summarizes basic properties of the relation $\prec$.

%==========================================================================================
\begin{lma}
Let $A$ be a \ca.
Then the following statements hold: %relation $\prec$ satisfies:
\begin{enumerate}
\item
Let $p,q,r,s\in\OP(A)$ satisfy $p\leq q\prec r\leq s$.
Then $p\prec s$.
\item
Let $p,q\in\OP(A)$ satisfy $p\prec q$.
Then there exists $q'\in\OP(A)$ with $p\prec q'\prec q$.
\end{enumerate}
\end{lma}
\begin{proof}
Statement~(1) is obvious.
To verify~(2), let $a\in A_+$ satisfy $p\leq a\leq q$.
Then $p=pa$ and $a=aq$.
Let $f,g\colon\RR\to[0,1]$ be continuous functions such that~$f$ takes the value $0$ on $[0,\tfrac{1}{2}]$ and the value $1$ on $[1,\infty)$, while $g$ takes the value $0$ on $\{0\}$ and the value $1$ on $[\tfrac{1}{2},\infty)$.
Note that $fg=f$.
Set $q'=\supp(f(a))\in\OP(A)$.
Then $p\leq f(a)\leq q'\leq g(a)\leq q$, which shows that $p\prec q'\prec q$.
\end{proof}

%==========================================================================================
The following definition is inspired by the notion of paths in $\mathcal{Q}$-semigroups (certain directed complete, partially ordered, abelian semigroups equipped with an auxiliary relation) from \cite[Paragraph~2.12]{AntPerThi20CuntzUltraproducts}.

%==========================================================================================
\begin{dfn}
\label{dfn:path}
A \emph{path} in $\OP(A)$ is an order-preserving map $f\colon(-\infty,0]\to\OP(A)$ satisfying the following conditions:
\begin{enumerate}
\item
$f(t)=\sup\{f(t'):t'<t\}$ for every $t\in(-\infty,0]$;
\item
$f(t')\prec f(t)$ for all $t'<t$ in $(-\infty,0]$.
\end{enumerate}
We say that a path $f$ is \emph{bounded} if there exist $t\in(-\infty,0]$ such that $f(t)=0$.
We then define the length of $f$ as $l(f):=\sup\{|t|:f(t)\neq 0\}$.
\end{dfn}

%==========================================================================================
Every positive element in $A$ defines a bounded path in $\OP(A)$:

%==========================================================================================
\begin{lma}
\label{prp:pathFromElement}
Let $A$ be a \ca{} and $a\in A_+$.
Define $f_a\colon(-\infty,0]\to\OP(A)$ by
\[
f_a(t) := \supp( (a+t)_+ )
\]
for $t\in(-\infty,0]$.
Then $f_a$ is a bounded path in $\OP(A)$ with length $l(f_a)=\|a\|$.
\end{lma}

%==========================================================================================
Below, we show that every bounded path in $\OP(A)$ is induced by an element in~$A_+$.
It follows that positive elements in $A$ correspond to bounded paths in $\OP(A)$.
However, the (pointwise) order on paths does not correspond to the usual order on $A_+$, but to the \emph{spectral order} introduced by Olsen in \cite{Ols71SelfadjLattice}.
We note that the next result is closely related to \cite[Theorem~I.1]{Ake71GelfandRepr}.

%==========================================================================================
\begin{prp}
\label{prp:elementFromPath}
Let $A$ be a \ca{} and let $f\colon(-\infty,0]\to\OP(A)$ be a bounded path.
Then there exists a unique positive element $a\in A$ such that $f=f_a$, that is, $f(t)=\supp((a+t)_+)$ for every $t\in(-\infty,0]$.
\end{prp}
\begin{proof}
\emph{Uniqueness:}
Let $a,b\in A_+$.
By \cite[Theorem~3]{Ols71SelfadjLattice}, $a$ is dominated by $b$ in the spectral order if and only if $a^n\leq b^n$ for every $n\geq 1$.
In particular, $f_a\leq f_b$ implies $a\leq b$, and $f_a=f_b$ implies $a=b$.

\emph{Existence:}
Set $\ZZ_{\leq 0}:=\ZZ\cap(-\infty,0]$. %and set $l:=l(f)$, the length of $f$.
Let $n\geq 1$.
For each $k\in\ZZ_{\leq 0}$, using that $f(\tfrac{k-1}{2^n})\prec f(\tfrac{k}{2^n})$, we obtain $a_{n,k}\in A_+$ such that
\[
f(\tfrac{k-1}{2^n})
\leq a_{n,k}
\leq  f(\tfrac{k}{2^n}).
\]
For $t\leq \tfrac{k-1}{2^n}$, we have $f(t)\leq a_{n,k}$.
By \autoref{prp:charCompPosProj}, this implies that $a_{n,k}$ commutes with $f(t)$.
Similarly, we obtain that $a_{n.k}$ commutes with $f(t)$ for every $t\geq \tfrac{k}{2^n}$.
In particular, $a_{n,k}$ commutes with $f(\tfrac{l}{2^n})$ for every $l\in\ZZ_{\leq 0}$.
We set
\[
a_n := \sum_{k=-\infty}^0 \tfrac{1}{2^n} a_{n,k},
\]
which is a finite sum since $a_{n,k}=0$ for $k<2^n l(f)$, where $l(f)$ is the lenght of $f$.
We deduce that $a_n$ commutes with $f(\tfrac{l}{2^n})$ for every $l\in\ZZ_{\leq 0}$.

For each $m\geq n$, we have
\begin{align}
\label{eq:elementFromPath}
\sum_{k\in\ZZ_{\leq 0}} \tfrac{1}{2^n} f(\tfrac{k-1}{2^n})
\leq a_m 
\leq \sum_{k\in\ZZ_{\leq 0}} \tfrac{1}{2^n} f(\tfrac{k}{2^n}).
\end{align}
Hence, $\|a_n-a_{n+1}\|\leq\tfrac{1}{2^n}$ for each $n$, and it follows that $(a_n)_n$ is a Cauchy sequence.
We set $a:=\lim_n a_n$, which is a positive element in $A$.

Let $n\geq 1$.
It follows from \eqref{eq:elementFromPath} that
\[
\sum_{j=-\infty}^0 \tfrac{1}{2^n} f(\tfrac{j-1}{2^n})
\leq a
\leq \sum_{j=-\infty}^0 \tfrac{1}{2^n} f(\tfrac{j}{2^n}).
\]
For each $j$, since $a_m$ commutes with $f(\tfrac{j}{2^n})$ for $m\geq n$, it follows that $a$ commutes with $f(\tfrac{j}{2^n})$.
Consequently, $a$ commutes with $\sum_{j=-\infty}^0 \tfrac{1}{2^n} f(\tfrac{j-1}{2^n})$.
Given \emph{commuting} self-adjoint elements $x,y$ in a \ca{} that satisfy $x\leq y$, and given $t\in\RR$, it follows that $(x+t)_+\leq(y+t)_+$.
(For noncommuting elements, this does not necessarily hold.)
Given $k\in\ZZ_{\leq 0}$, we have
\[
\left( \left( \sum_{j=-\infty}^0 \tfrac{1}{2^n} f(\tfrac{j}{2^n}) \right) - \tfrac{k}{2^n} \right)_+
= \sum_{j=-\infty}^{k-1} \tfrac{1}{2^n} f(\tfrac{j}{2^n}),
\]
and therefore
\[
\sum_{j=-\infty}^{k-1} \tfrac{1}{2^n} f(\tfrac{j}{2^n})
\leq (a-\tfrac{k}{2^n})_+
\leq \sum_{j=-\infty}^{k} \tfrac{1}{2^n} f(\tfrac{j}{2^n}).
\]

It follows that
\[
f(\tfrac{k-1}{2^n}) 
%= \supp\left( \sum_{j=-\infty}^{k-1} \tfrac{1}{2^n} f(\tfrac{j}{2^n}) \right) 
\leq \supp( (a-\tfrac{k}{2^n})_+ )
\leq f(\tfrac{k}{2^n})
\]
for every $n\geq 1$ and every $k\in\ZZ$, $k\leq 0$.
Using condition~(1) in the definition of a path, we obtain
\[
f(\tfrac{k}{2^n})
= \sup_{l\geq 1} f(\tfrac{2^l k - 1}{2^{n+l}}) \leq \supp( (a-\tfrac{k}{2^n})_+ ),
\]
and so $\supp( (a-\tfrac{k}{2^n})_+ ) = f(\tfrac{k}{2^n})$ for every $n\geq 1$ and every $k\in\ZZ$, $k\leq 0$.
It follows that $f$ agrees with $f_a$ at every nonpositive dyadic number.
Since both $f$ and $f_a$ are paths, we deduce that $f=f_a$.
\end{proof}

%==========================================================================================
%==========================================================================================
\section{Diffuse and nowhere scattered functionals}
\label{sec:NWS}

%==========================================================================================
In this section, we first recall basic properties of diffuse and atomic functionals.
We then introduce the main technical concept of this paper:
a positive functional is \emph{nowhere scattered} if it gives no weight to scattered ideal-quotients;
see \autoref{dfn:NWS}.
Equivalently, the functional gives no weight to elementary ideal-quotients, or it vanishes on minimal projections in quotients of the algebra;
see \autoref{prp:firstCharNWS}.
(A projection $p$ in a \ca{} $A$ is said to be minimal it $p\neq 0$ and $pAp=\CC p$.)
In the next section, we prove that nowhere scattered functionals admit Haar unitaries.

We show that every diffuse functional is nowhere scattered;
see \autoref{prp:diffuseImpliesNWS}.
A \ca{} is type~$\mathrm{I}$ if and only if every of its nowhere scattered states is diffuse;
see \autoref{prp:typeI}.
In forthcoming work, we will study the class of \ca{s} with the property that every positive functional is nowhere scattered.

%==========================================================================================
\begin{pgr}
\label{pgr:diffuse}
Let $A$ be a \ca.
By a \emph{positive functional} on $A$ we mean a positive, linear map $\varphi\colon A\to\CC$.
Every positive functional is bounded and therefore extends uniquely to a normal, positive functional $A^{**}\to\CC$ which we also denote by~$\varphi$.

A positive functional $\varphi$ is \emph{diffuse} if $\varphi(e)=0$ for every minimal projection $e$ in $A^{**}$.
Equivalently, $\varphi(z_{\mathrm{at}})=0$, where $z_{\mathrm{at}}\in A^{**}$ denotes the supremum of all minimal projections in $A^{**}$.
A positive functional $\varphi$ is \emph{atomic} if $\varphi(1-z_{\mathrm{at}})=0$.
The notions of atomic and diffuse functionals on a \ca{} were introduced by Pedersen in \cite{Ped71AtomicDiffuse} using the concept of Baire operators.
It is straightforward to verify that the definitions in \cite{Ped71AtomicDiffuse} are equivalent to the ones above, and also equivalent to \cite[Definition~1.1]{Jen77ScatteredCAlg}.

By \cite[Proposition~4]{Ped71AtomicDiffuse}, a positive functional $\varphi$ is atomic if and only if $\varphi=\sum_{k=1}^\infty \alpha_k \varphi_k$ for a sequence $(\varphi_k)_k$ of pure states and positive coefficients $(\alpha_k)_k$ with $\sum_k\alpha_k<\infty$;
see also \cite[Theorem~1.2]{Jen77ScatteredCAlg}.

Every positive functional $\varphi$ admits a unique decomposition as a sum of an atomic, positive functional $\varphi_{\mathrm{a}}$ and a diffuse, positive functional $\varphi_{\mathrm{d}}$.
With $z_{\mathrm{at}}\in A^{**}$ as above, we have
\[
\varphi_{\mathrm{a}}(a)=\varphi(az_{\mathrm{at}}), \andSep
\varphi_{\mathrm{d}}(a)=\varphi(a(1-z_{\mathrm{at}}))
\]
for $a\in A$.
It follows that a positive functional is diffuse if and only if it does not dominated a nonzero multiple of a pure state.
\end{pgr}

%==========================================================================================
In the next result, we use $\pi_\varphi\colon A\to\Bdd(H_\varphi)$ to denote the GNS-representation associated to a positive functional $\varphi$, and we let $\pi_\varphi(A)''\subseteq\Bdd(H_\varphi)$ denote the generated von Neumann algebra.
Recall that a von Neumann algebra is \emph{diffuse} if it contains no minimal projections.
It is called \emph{atomic} (sometimes `purely atomic') if its unit is the supremum of minimal projections; equivalently, it is a product of type~$\mathrm{I}$~factors.

%==========================================================================================
\begin{lma}
\label{prp:diffuseGNS}
Let $\varphi\colon A\to\CC$ be a positive functional on a \ca{} $A$.
Then $\varphi$ is diffuse (atomic) if and only if $\pi_\varphi(A)''$ is diffuse (atomic).
\end{lma}
\begin{proof}
Let $s_\varphi$ denote the support projection of $\varphi$ in $A^{**}$, and let $c(s_\varphi)$ denote its central cover.
We have
\[
\pi_\varphi(A)'' \cong c(s_\varphi)A^{**};
\]
see for example \cite[III.2.2.23f]{Bla06OpAlgs}.
Let $z_{\mathrm{at}}\in A^{**}$ denote the supremum of all minimal projections in $A^{**}$.

Now $\varphi$ is diffuse if and only if $\varphi(z_{\mathrm{at}})=0$, which is equivalent to $s_\varphi\leq 1-z_{\mathrm{at}}$.
Since $z_{\mathrm{at}}$ is central, this is also equivalent to $c(s_\varphi)\leq 1-z_{\mathrm{at}}$.
Given a central projection $z\in A^{**}$, note that $zA^{**}$ is diffuse if and only if $z\leq 1-z_{\mathrm{at}}$.
Thus, $\varphi$ is diffuse if and only if $c(s_\varphi)A^{**}$ is diffuse.
The atomic case is proved analogously.
\end{proof}

%==========================================================================================
\begin{exa}
\label{exa:diffuseMeasure}
Let $X$ be a locally compact, Hausdorff space.
Positive functionals on $C_0(X)$ naturally correspond to bounded, positive Borel measures on $X$.
Given a measure $\mu$, the corresponding functional $\varphi_\mu\colon C_0(X)\to\CC$ is given by $\varphi_\mu(f)=\int_X f(x)d\mu(x)$.
Then $\varphi_\mu$ is diffuse if and only if $\mu$ is diffuse, that is, has not atoms. 
%Similarly, $\varphi_\mu$ is atomic if and only if $\mu$ is atomic, that is, has not atoms.
\end{exa}

%==========================================================================================
%\begin{lma}
%\label{prp:diffuse-RestrictHer}
%Let $A$ be a \ca{}, let $B\subseteq A$ be a hereditary sub-\ca{}, and let $\varphi\colon A\to\CC$ be a diffuse, positive functional.
%Then $\varphi|_{B}$ is diffuse.
%\end{lma}
%\begin{proof}
%Let $p\in A^{**}$ be the open projection corresponding to $B$.
%We naturally identify $B^{**}$ with $pA^{**}p$.
%It follows that every minimal projection in $B^{**}$ is also minimal in $A^{**}$, which implies that $\varphi|_{B}$ is diffuse.
%\end{proof}

%==========================================================================================
\begin{pgr}
A \ca{} $A$ is \emph{scattered} if every positive functional on $A$ is atomic;
see \cite[Definition~2.1]{Jen77ScatteredCAlg}.
This is known to be equivalent to many other properties.
For example, $A$ is scattered if and only if every self-adjoint element in $A$ has countable spectrum, if and only if $A$ has a composition series $(K_\alpha)_{0\leq\alpha\leq\beta}$ such that the successive quotients $K_{\alpha+1}/K_\alpha$ are elementary (a \ca{} is \emph{elementary} if it is isomorphic to the algebra of compact operators on some Hilbert space);
see \cite[Theorem~1.4]{GhaKos18NCCantorBendixson} and \cite[Theorem~2]{Jen78ScatteredCAlg2}.
\end{pgr}

%==========================================================================================
Given a \ca{} $A$, an \emph{ideal-quotient} of $A$ is a (closed, two-sided) ideal of a quotient of $A$.
Using the correspondence between ideals (quotients) of a \ca{} and open (closed) subsets of its primitive ideal space, it follows that ideal-quotients of $A$ correspond to locally closed subsets of the primitive ideal space of $A$.

%==========================================================================================
\begin{dfn}
\label{dfn:NWS}
A positive functional $\varphi\colon A\to\CC$ on a \ca{} $A$ is \emph{nowhere scattered} if $\|\varphi|_I\|=\|\varphi|_J\|$ for all ideals $I\subseteq J\subseteq A$ such that $J/I$ is scattered.
\end{dfn}

%==========================================================================================
\begin{rmk}
\label{rmk:NWS}
Let $\varphi\colon A\to\CC$ be a positive functional on a \ca{} $A$.
Given ideals $I\subseteq J\subseteq A$, let $p,q\in\OP(A)$ be the corresponding central open projections.
Then $\|\varphi|_I\|=\varphi(p)$ and $\|\varphi|_J\|=\varphi(q)$.
Thus, we have $\|\varphi|_I\|=\|\varphi|_J\|$ if and only if $\varphi(p)=\varphi(q)$, that is, $\varphi(q-p)=0$.
\end{rmk}

%==========================================================================================
\begin{prp}
\label{prp:firstCharNWS}
Let $A$ be a \ca{}, and let $\varphi\colon A\to\CC$ be a positive functional.
Then the following are equivalent:
\begin{enumerate}
\item
$\varphi$ is nowhere scattered;
\item
$\|\varphi|_I\|=\|\varphi|_J\|$ for all ideals $I\subseteq J\subseteq A$ such that $J/I$ is elementary;
\item
for every ideal $I\subseteq A$ and every minimal projection $e\in A/I$, we have $\varphi(e)=0$ (viewing $e\in A/I \subseteq (A/I)^{**} \subseteq A^{**}$).
\end{enumerate}
\end{prp}
\begin{proof}
To show that~(3) implies~(2), let $I\subseteq J\subseteq A$ be ideals such that $J/I$ is elementary.
Let $p,q\in\OP(A)$ denote the central, open projections corresponding to~$I$ and $J$.
We can choose an approximate unit of finite-rank projections $(e_\lambda)_\lambda$ in~$J/I$.
It follows from the assumption that $\varphi(e_\lambda)=0$ for each $\lambda$.
We get
\[
\varphi(q-p)=\varphi\big( \sup_\lambda e_\lambda \big) = 0.
\]

To show that~(2) implies~(1), let $I\subseteq J\subseteq A$ be ideals such that $J/I$ is scattered.
Let $\pi\colon J\to J/I$ denote the quotient map.
By \cite[Theorem~2]{Jen78ScatteredCAlg2}, there exists a composition series $(K_\alpha)_{0\leq\alpha\leq\beta}$ for $J/I$ such that $K_{\alpha+1}/K_\alpha$ is elementary for each $\alpha<\beta$.
In particular, $K_0=\{0\}$ and $K_\beta=J/I$.
For each $\alpha\leq\beta$, let $p_\alpha\in\OP(A)$ be the central open projection corresponding to the ideal $\pi^{-1}(K_\alpha)$.

Using transfinite induction, we show that $\varphi(p_\alpha-p_0)=0$ for each $\alpha\leq\beta$.
This is clear for $\alpha=0$.
Assuming that it holds for some $\alpha$, we use that $K_{\alpha+1}/K_\alpha$ is elementary and thus $\varphi(p_{\alpha+1}-p_\alpha)=0$ to deduce that $\varphi(p_{\alpha+1}-p_0)=0$.
If $\alpha$ is a limit ordinal and we have $\varphi(p_{\alpha'}-p_0)=0$ for every $\alpha'<\alpha$, then
\[
\varphi(p_{\alpha}-p_0)
= \varphi\left( \sup_{\alpha'<\alpha}(p_{\alpha'}-p_0) \right)
= \sup_{\alpha'<\alpha}\varphi(p_{\alpha'}-p_0)
= 0.
\]
Thus, $\varphi(p_{\beta}-p_0)=0$, and so
\[
\|\varphi|_I\|
= \varphi(p_0) 
= \varphi(p_\beta)
= \|\varphi|_J\|.
\]

To show that~(1) implies~(3), let $I\subseteq A$ be an ideal, and let $e\in A/I$ be a minimal projection.
Let $\pi\colon A\to A/I$ denote the quotient map.
Let $K\subseteq A/I$ denote the sub-\ca{} generated by all minimal projections in $A/I$.
By \cite[Theorem~1.2]{GhaKos18NCCantorBendixson}, $K$ is scattered and an ideal of $A/I$.
Thus, if $p$ and $q$ denote the central, open projections corresponding to the ideals $I$ and $\pi^{-1}(K)$, then $\varphi(q-p)=0$.
Since $e\leq q-p$, it follows that $\varphi(e)=0$.
\end{proof}

%==========================================================================================
\begin{lma}
\label{prp:NWS-RestrictHer}
Let $A$ be a \ca{}, let $B\subseteq A$ be a hereditary sub-\ca{}, and let $\varphi\colon A\to\CC$ be a nowhere scattered, positive functional.
Then $\varphi|_{B}$ is nowhere scattered.
\end{lma}
\begin{proof}
We use the characterization of unscattered functionals through minimal projections in quotients from condition~(3) in \autoref{prp:firstCharNWS}.
Let $I\subseteq B$ be an ideal, and let $e\in B/I$ be a minimal projection.
We need to show that $\varphi|_B(e)=0$.
Let $J\subseteq A$ be the ideal of $A$ generated by $I$, and let $\pi\colon A\to A/J$ denote the quotient map.
We have $I=B\cap J$, and so $\pi(B)$ is a hereditary sub-\ca{} of $A/J$ and $\pi(B)\cong B/I$.
Thus, $e$ is also a minimal projection of $A/J$.
Since $\varphi$ is nowhere scattered, we get $\varphi(e)=0$, and so $\varphi|_B(e)=0$.
\end{proof}

%==========================================================================================
\begin{prp}
\label{prp:diffuseImpliesNWS}
Every diffuse, positive functional is nowhere scattered.
\end{prp}
\begin{proof}
Let $A$ be a \ca{}, and let $\varphi\colon A\to\CC$ be a diffuse, positive functional.
Let $I\subseteq A$ be an ideal, and let $e\in A/I$ be a minimal projection.
Then $e$ is also a minimal projection in $(A/I)^{**}$, and thus in $A^{**}$.
Since $\varphi$ is diffuse, we get $\varphi(e)=0$.
By \autoref{prp:firstCharNWS}, it follows that $\varphi$ is nowhere scattered.
\end{proof}

%==========================================================================================
The converse of \autoref{prp:diffuseImpliesNWS} does not hold.
In fact, we will show that a \ca{} is not type~$\mathrm{I}$ if and only if it has a nowhere scattered state that is pure (and therefore not diffuse);
see \autoref{prp:typeI}.

A nowhere scattered functional is diffuse if it is also tracial (\autoref{prp:traceDiffuseNWS}) or if it is a normal functional on a von Neumann algbera (\autoref{prp:vNdiffuse}).

%==========================================================================================
\begin{lma}
\label{prp:pureNWS}
Let $\varphi$ be a pure state on a \ca{} $A$, and let $\pi_\varphi\colon A\to\Bdd(H_\varphi)$ be the induced GNS-representation.
Then $\varphi$ is nowhere scattered if and only if $\pi_\varphi(A)\cap\Cpct(H_\varphi)=\{0\}$.
\end{lma}
\begin{proof}
Set $L:=\ker(\pi_\varphi)$.
The restriction of a pure state to an ideal is either a pure state or zero.
Thus, given an ideal $I\subseteq A$, we either have $\|\varphi|_I\|=1$ (which happens precisely if $I$ is not contained in $L$) or $\varphi|_I=0$ (which happens if and only if $I\subseteq L$).

To show that forward implication, assume that $\pi_\varphi(A)\cap\Cpct(H_\varphi)\neq\{0\}$.
Since $\pi_\varphi$ is irreducible, it follows that $\Cpct(H_\varphi)\subseteq\pi_\varphi(A)$;
see \cite[Corollary~IV.1.2.5]{Bla06OpAlgs}
Set $J:=\pi_\varphi^{-1}(\Cpct(H_\varphi))$.
Then $J/L\cong \Cpct(H_\varphi)$ is elementary and $\|\varphi|_L\|=0<1=\|\varphi|_J\|$, showing that $\varphi$ is not nowhere scattered.

To show the converse implications, assume that $\varphi$ is not nowhere scattered.
Choose ideals $I\subseteq J\subseteq A$ such that $J/I$ is scattered and such that $\|\varphi|_I\|<\|\varphi|_J\|$.
This forces $\|\varphi|_J\|=1$, and consequently $J$ is not contained in $L$.
It follows that the restriction of $\pi_\varphi$ to $J$ is a nonzero, irreducible representation.
Since $J$ is scattered, and therefore of type~$\mathrm{I}$, it follows that $\pi_\varphi(J)$ contains a nonzero, compact operator.
Hence, $\pi_\varphi(A)\cap \Cpct(H_\varphi)\neq\{0\}$.
\end{proof}

%==========================================================================================
\begin{lma}
\label{prp:dominatedNWS}
Let $\varphi$ and $\psi$ be positive functionals on a \ca{} satisfying $\psi\leq\varphi$.
Assume that $\varphi$ is nowhere scattered.
Then $\psi$ is nowhere scattered.
\end{lma}
\begin{proof}
Let $I\subseteq J\subseteq A$ be ideals such that $J/I$ is scattered.
Let $p$ and $q$ be the central, open projections corresponding to $I$ and $J$, respectively.
Then
\[
\|\psi|_J\| - \|\psi|_I\|
= \psi(q-p)
\leq \varphi(q-p)
= \|\varphi|_J\| - \|\varphi|_I\|
= 0,
\]
as desired.
See also \autoref{rmk:NWS}.
%However, it follows directly from the definition that a positive functional is nowhere scattered if it is dominated by a multiple of a nowhere scattered functional.
\end{proof}

%==========================================================================================
\begin{prp}
\label{prp:typeI}
Let $A$ be a \ca.
Then $A$ is of type~$\mathrm{I}$ if and only if every nowhere scattered state on $A$ is diffuse.
\end{prp}
\begin{proof}
To prove the forward implication, assume that $A$ is of type~$\mathrm{I}$.
By \autoref{prp:diffuseImpliesNWS}, it remains to verify that every nowhere scattered functional is diffuse.
To prove the contraposition, let $\varphi$ be a positive functional on $A$ that is not diffuse.
Then there exists a pure state $\psi$ and $t>0$ such that $\psi\leq t\varphi$;
see \autoref{pgr:diffuse}.
Since $A$ is type~$\mathrm{I}$, we have $\pi_\psi(A)\cap\Cpct(H_\psi)\neq\{0\}$ and it follows from \autoref{prp:pureNWS} that $\psi$ is not nowhere scattered.
By \autoref{prp:dominatedNWS}, neither is $t\varphi$, which implies that~$\varphi$ is not nowhere scattered.

To show that backward implication, assume that $A$ is not of type~$\mathrm{I}$.
Then there exists a pure state $\varphi$ such that $\pi_\varphi$ is not GCR, that is, 
$\pi_\varphi(A)\cap\Cpct(H_\varphi)=\{0\}$.
By \autoref{prp:pureNWS}, $\varphi$ is nowhere scattered, yet not diffuse.
\end{proof}

%==========================================================================================
%==========================================================================================
\section{Haar unitaries characterize nowhere scattered functionals}

%==========================================================================================
The main result of this section is \autoref{prp:charNWS}, which provides several characterizations for a functional to be nowhere scattered.
Most interestingly, a positive functional $A\to\CC$ is nowhere scattered if and only if the minimal unitization of every hereditary sub-\ca{} of $A$ contains a Haar unitary.
It follows in particular that every positive functional on a unital, simple, infinite-dimensional \ca{} admits a Haar unitary;
see \autoref{prp:HaarSimple}.

%==========================================================================================
Every positive functional $\varphi\colon A\to\CC$ extends uniquely to a normal, positive functional $A^{**}\to\CC$, which we also denote by $\varphi$. 
The induced map $\OP(A)\to[0,\infty)$, $p\mapsto\varphi(p)$, is order-preserving and satisfies $\varphi(\sup_j p_j)=\sup_j \varphi(p_j)$ for every increasing net $(p_j)_j$ in $\OP(A)$.
Given $p\in\OP(A)$, it is easy to see that $p$ is the supremum of the set $\{ p'\in\OP(A) : p'\prec p \}$.
However, this set is not necessarily upward-directed.
Nevertheless, we have:

%==========================================================================================
\begin{prp}
\label{prp:cutDown}
Let $\varphi\colon A\to\CC$ be a positive functional, let $p,q\in\OP(A)$ satisfy $p\prec q$, and let $\varepsilon>0$.
Then there exists $q'\in\OP(A)$ such that
\[
p\prec q'\prec q, \andSep
\varphi(q)-\varepsilon \leq \varphi(q').
\]
In particular, we have
\[
\varphi(q) = \sup \big\{ \varphi(q') : q'\in\OP(A), q'\prec q \big\}.
\]
\end{prp}
\begin{proof}
Choose $a\in A_+$ such that $p\leq a\leq q$.
Let $C^*(a,q)$ denote the sub-\ca{} of $A^{**}$ generated by $a$ and $q$, and let $\spec(a)$ denote the spectrum of $a$, which is a closed subset of $[0,1]$.
There is a natural isomorphism between $C^*(a,q)$ and $C(\spec(a))$.
The restriction of $\varphi$ to $C^*(a,q)$ induces a measure $\mu_\varphi$ on $\spec(a)$, which we view as a measure on $[0,1]$.
Choose $s,t$ such that
\[
0\leq s < t \leq 1, \andSep \mu_\varphi((s,t)) <\varepsilon.
\]
Then choose continuous functions $f,g,h\colon[0,1]\to[0,1]$ that take the value $0$ on $[0,s]$, that take the value $1$ on $[t,1]$, and that satisfy $f=fg$ and $g=gh$.
The elements $f(a)$, $g(a)$ and $h(a)$ belong to $A_+$ and satisfy 
\[
p\leq f(a), \quad
f(a) = f(a) g(a), \andSep
g(a) = g(a) h(a) \leq q.
\]

Set $z:=\supp(q-h(a))\in A^{**}$.
To see that $z$ is an open projection, let $(b_j)_j$ be an increasing net in $A_+$ with $q=\sup_j b_j$.
Then $((q-h(a))b_j(q-h(a))_j$ is an increasing net in $A_+$ with supremum $(q-h(a))^2$.
Then
\[
z 
= \supp\big( q-h(a) \big)
= \supp\big( (q-h(a))^2 \big)
= \sup_j \supp\big( (q-h(a))b_j(q-h(a)) \big),
\]
which shows that $z\in\OP(A)$.
We have $q-h(a) \leq z$ and therefore
\[
\mu_\varphi([t,1])\leq\varphi(q-h(a))\leq\varphi(z).
\]

Set $w:=\supp(f(a)) \in \OP(A)$.
We have $f(a)\leq w$ and therefore
\[
\mu_\varphi([0,s])\leq\varphi(f(a))\leq\varphi(w).
\]

Note that $f(a)=f(a)g(a)$ implies that $w=wg(a)$.
Similarly, it follows from $q-h(a)=(q-g(a))(q-h(a))$ that $z=(q-g(a))z$.
Hence,
\[
wz
= w(q-g(a))z
= 0.
\]
We have $\mu_\varphi([0,1])=\varphi(q)$ and therefore
\[
\varphi(w+z)
\geq \mu_\varphi([0,s]\cup[t,1])
= \mu_\varphi([0,1])-\mu_\varphi((s,t))
> \varphi(q) - \varepsilon.
\]

Let $(c_j)_j$ be an increasing net in $A_+$ with supremum $z$.
Using that $\varphi(z)=\sup_j\varphi(c_j)$, choose $j$ such that
\[
\varphi(q) - \varepsilon
< \varphi(w) + \varphi(c_j).
\]

For each $\delta>0$ let $(c_j-\delta)_+$ be the element obtained by applying functional calculus for the function $t\mapsto\max\{0,t-\delta\}$ to $z_j$.
Using that $c_j=\sup_{\delta>0}(c_j-\delta)_+$, choose $\delta>0$ such that
\[
\varphi(q) - \varepsilon
< \varphi(w) + \varphi\big( (c_j-\delta)_+ \big).
\]

Let $e\colon[0,1]\to[0,1]$ be a continuous function with $e(0)=0$ and taking the value~$1$ on $[\delta,1]$.
Then
\[
\supp\big( (z_j-\delta)_+ \big)
\leq e\big( (z_j-\delta)_+ \big)
\leq z,\quad
w\leq g(a), \andSep
zg(a)=0.
\]
Thus, $q':=\supp((z_j-\delta)_+) + w$ is an open projection satisfying
\[
q' \leq e\big( (z_j-\delta)_+ \big) + g(a) \leq q
\]
and thus $q'\prec q$.
Further, we have
\[
p\leq f(a) \leq w \leq q'
\]
and thus $p\prec q'$.
Lastly, we have
\[
\varphi(q) - \varepsilon
< \varphi(w) + \varphi\big( (c_j-\delta)_+ \big)
\leq \varphi(w) + \varphi\big( \supp((c_j-\delta)_+) \big)
= \varphi(q'),
\]
as desired.
\end{proof}

%==========================================================================================
Given a positive functional $\varphi\colon A\to\CC$, it is well-known that the set
\[
L_\varphi := \big\{ a\in A : \varphi(a^*a)=0 \big\}
\]
is a closed, left ideal;
see \cite[Theorem~3.3.3]{Ped79CAlgsAutGp}.
Since $a\in A$ satisfies $\varphi(a^*a)=0$ if and only if $\varphi(\supp(a^*a))=0$, we have $L_\varphi := \big\{ a\in A : \varphi(\supp(a^*a))=0 \big\}$.

Analogously, $L_\varphi^*=\{ a\in A : \varphi(\supp(aa^*))=0 \}$ is a closed, right ideal, and
\[
L_\varphi\cap L_\varphi^* = \big\{ a\in A : \varphi(\supp(a^*a)) =\varphi(\supp(aa^*)) =0 \big\}
\]
is a hereditary sub-\ca{} of $A$.
The next result shows that $L_\varphi\cap L_\varphi^*$ is even a (closed, two-sided) ideal if $0$ is isolated in $\varphi(\OP(A))$.
It actually suffices that $\varphi(\OP(A))$ contains a sufficiently large gap near zero.

%==========================================================================================
\begin{lma}
\label{prp:gapGivesIdeal}
Let $A$ be a \ca{}, let $\varphi\colon A\to\CC$ be a nonzero, positive functional, and let $\delta\in(0,\|\varphi\|)$ such that $(\delta,2\delta]\cap\varphi(\OP(A))=\emptyset$.
Then
\[
I := \big\{ a\in A : \varphi(\supp(aa^*)),\varphi(\supp(a^*a))\leq\delta \big\}
\]
is an ideal in $A$.
\end{lma}
\begin{proof}
Set
\[
U := \big\{ p\in\OP(A) : \varphi(p)\leq\delta \big\}, \andSep
z:=\bigvee U\in\OP(A).
\]
Given $p,q\in U$, we have %$p\vee q\leq p+q$ and therefore
\[
\varphi(p\vee q)\leq\varphi(p+q)\leq 2\delta.
\]
The assumption on the gap in $\varphi(\OP(A))$ implies that $p\vee q$ belongs to $U$.
Thus, $U$ is upward-directed and using that $\varphi\colon A^{**}\to\CC$ is a normal functional, we get
\[
\varphi(z)=\varphi(\bigvee U)=\sup_{p\in U}\varphi(p)\leq\delta.
\]
Thus, $z$ is the largest element in $U$.
It follows that $I$ is the hereditary sub-\ca{} of $A$ corresponding to $z$.

To show that $I$ is a two-sided ideal, let $\mathcal{U}_0(\widetilde{A})$ denote the subgroup of unitaries in the minimal unitization of $A$ that are connected to the unit.
Given $a\in I$ and $u\in\mathcal{U}_0(\widetilde{A})$, choose a continuous path $[0,1]\to\mathcal{U}_0(\widetilde{A})$, $t\mapsto u_t$, with $u_0=1$ and $u_1=u$.
For $t\in[0,1]$, we have
\[
\supp(u_taa^*u_t)=u_t\supp(aa^*)u_t.
\]
Thus, $t\mapsto\supp(u_taa^*u_t)$ is a continuous path of open projections in $A^{**}$, and so $t\mapsto\varphi(\supp(u_taa^*u_t))$ is a continuous map $[0,1]\to\RR$.
The gap in $\varphi(\OP(A))$ implies that
\[
\varphi(\supp(uaa^*u))\leq\delta.
\]
Since also $\varphi(\supp(a^*uua))=\varphi(\supp(a^*a))\leq\delta$, we deduce that $ua$ belongs to $I$.
Using that every element in $\widetilde{A}$ is a finite linear combination of elements in $\mathcal{U}_0(\widetilde{A})$, it follows that $I$ is a left ideal.
Analogously, we obtain that $I$ is a right ideal.
\end{proof}

%==========================================================================================
\begin{lma}
\label{prp:gapGivesFDQuotient}
Let $A$ be a \ca{}, let $\varphi\colon A\to\CC$ be a nonzero, positive functional, and let $\delta\in(0,\|\varphi\|)$ such that $(\delta,2\delta]\cap\varphi(\OP(A))=\emptyset$.
Then there exists an ideal $I\subseteq A$ such that $A/I$ is finite-dimensional and $\|\varphi|_I\|<\|\varphi\|$. 
\end{lma}
\begin{proof}
Set
\[
I := \big\{ a\in A : \varphi(\supp(aa^*)),\varphi(\supp(a^*a))\leq\delta \big\},
\]
which is an ideal in $A$ by \autoref{prp:gapGivesIdeal}.
Let $z\in A^{**}$ denote the corresponding central, open projection.
By assumption, we have $\delta<\|\varphi\|$ and therefore
\[
\|\varphi|_I\|
= \varphi(z)
\leq \delta 
< \|\varphi\|.
\]

Set $\delta_0:=\varphi(z)$.
Then $(\delta_0,2\delta]\cap\varphi(\OP(A))=\emptyset$.
Set $B:=A/I$, and let $\pi\colon A\to B$ denote the quotient map.
Using the natural identification of $B^{**}$ with $(1-z)A^{**}$, we define $\psi\colon B\to\CC$ as the composition
\[
B\subseteq B^{**} \cong (1-z)A^{**}\subseteq A^{**}\xrightarrow{\varphi}\CC.
\]

\emph{Claim: The functional $\psi\colon B\to\CC$ is faithful.}
To prove the claim, let $b\in B_+$ satisfy $\psi(b)=0$.
Lift $b$ to find $a\in A_+$ with $\pi(a)=b$.
Let $p\in\OP(A)$ be the support projection of $a$.
The natural isomorphism $(1-z)A^{**}\cong B^{**}$ identifies $(1-z)p$ with the support projection of $b$, and we get
\[
\varphi\big( (1-z)p \big)
= \psi\big( \supp(b) \big)
=0.
\]
Since $p\leq z+(1-z)p$, we obtain that
\[
\varphi(p)
\leq \varphi(z) + \varphi\big( (1-z)p \big)
= \delta_0 
\leq \delta.
\]
We get $a\in I$ and so $b=0$,
which proves the claim.

\emph{Claim: We have $(0,2\delta-\delta_0]\cap\psi(\OP(B))=\emptyset$.}
To prove the claim, assume that $p\in\OP(B)$ satisfies $\psi(p)\in(0,2\delta-\delta_0]$.
We identify $B^{**}$ with $(1-z)A^{**}\subseteq A^{**}$.
Set $\bar{p}:=z+p$.
Then $\bar{p}$ is the open projection in $A^{**}$ corresponding to the hereditary sub-\ca{} $\pi^{-1}(pB^{**}p\cap B)\subseteq A$.
We have
\[
\varphi(\bar{p})
= \varphi(z+p)
= \delta_0 + \psi(p)
\in (\delta_0,2\delta].
\]
This contradicts that $(\delta_0,2\delta]\cap\varphi(\OP(A))=\emptyset$.
The claim is proved.

It follows from the above claims that $\psi\colon B\to\CC$ is faithful and that $0$ is isolated in $\psi(\OP(B))$.
Let us show that this implies that every positive element in~$B$ has finite spectrum.
To reach a contradiction, let $b \in B_+$ have infinite spectrum~$\spec(b)$.
This allows us to choose a countable family $(f_k)_{k\in\NN}$ of nonzero, continuous functions $f_k \colon \spec(b) \to [0,1]$ such that $f_kf_l=0$ if $k \neq l$.
Then the open projections $\supp(f_k(b)) \in \OP(B)$ are pairwise orthogonal and satisfy $\supp(f_k(b)) \leq \supp(b)$, which implies that
\[
\sum_{k\in\NN} \psi(\supp(f_k(b))) \leq \psi(\supp(b)).
\]

On the other hand, since $0\neq f_k(b) \leq \supp(f_k(b))$, and since $\psi$ is faithful, we have $0 < \psi(\supp(f_k(b))$ for each $k$.
This implies that $\{\psi(\supp(f_k(b)):k\in\NN\}$ contains arbitrarily small positive elements, contradicting that $0$ is isolated in~$\psi(\OP(B))$.

Using that a \ca{} is finite-dimensional if (and only if) every of its positive elements has finite spectrum, we deduce that $B$ is finite-dimensional.
\end{proof}

%==========================================================================================
Given $p\in\OP(A)$, we use $[0,p]$ to denote the set $\{q\in\OP(A) : q\leq p\}$.

%==========================================================================================
\begin{lma}
\label{prp:NWSImpliesDenseOP}
Let $A$ be a \ca{}, and let $\varphi\colon A\to\CC$ be a nowhere scattered, positive functional.
Then $\varphi([0,p])$ is dense in $[0,\varphi(p)]$ for every $p\in\OP(A)$.
\end{lma}
\begin{proof}
Let $p\in\OP(A)$.
To reach a contradiction, assume that $\varphi([0,p])$ is not dense in $[0,\varphi(p)]$. 
Choose $t\in[0,\varphi(p))$ and $\delta>0$ satisfying $(t,t+2\delta)\cap\varphi([0,p])=\emptyset$.
We may assume that
\[
t = \sup \big\{ \varphi(p') : p'\in[0,p], \varphi(p')\leq t \big\},
\]
which allows us to choose $p'\in\OP(A)$ satisfying $p'\leq p$ and $t-\delta<\varphi(p')$.
Apply \autoref{prp:cutDown} to obtain $p''\in\OP(A)$ satisfying
\[
p''\prec p', \andSep t-\delta<\varphi(p''). %\leq\varphi(p')\leq t.
\]
Then $\overline{p''}\leq p'\leq p$, and we set $q:=p-\overline{p''}$, which is an open projection.
%Then $q$ is an open projection.
We have $\varphi(p)\geq t+2\delta$ and $\varphi(p'')\leq t$, and therefore
\[
\varphi(q)
= \varphi(p)-\varphi(\overline{p''})
\geq \varphi(p)-\varphi(p'')
\geq t+2\delta-t = 2\delta.
\]

\emph{Claim:
We have $(\delta,2\delta]\cap\varphi([0,q])=\emptyset$.}
To prove the claim, assume that an open projection $q'\leq q$ satisfies $\varphi(q')\in(\delta,2\delta]$.
Since $p''$ and $q$ are orthogonal and satisfy $p''+q\leq p$, we get that $p''+q'\in[0,p]$.
On the other hand,
\[
\varphi(p''+q')
= \varphi(p'') + \varphi(q') \in (t-\delta,t] + (\delta,2\delta] \subseteq (t,t+2\delta],
\]
which contradicts that $(t,t+2\delta]\cap\varphi([0,p])=\emptyset$.
This proves the claim.

Set $B:=qA^{**}q\cap A$, the hereditary sub-\ca{} of $A$ corresponding to $q$.
By construction, the restriction of $\varphi$ to $B$ satisfies the assumptions of \autoref{prp:gapGivesFDQuotient}.
Hence, $\varphi|_{B}$ is not nowhere scattered, which contradicts \autoref{prp:NWS-RestrictHer}.
\end{proof}

%==========================================================================================
\begin{lma}
\label{prp:interpolatingPairs}
Let $A$ be a \ca{}, let $\varphi\colon A\to\CC$ be a positive functional.
Assume that $\varphi([0,z])$ is dense in $[0,\varphi(z)]$ for every $z\in\OP(A)$.
Let $p,\tilde{p},q\in\OP(A)$ and $t,\tilde{t}\in[0,1]$ satisfy
\[
p\prec\tilde{p}\prec q, \andSep 
\varphi(\tilde{p})\leq t<\tilde{t}\leq\varphi(q).
\]
Then there exist $r,\tilde{r}\in\OP(A)$ satisfying
\[
p \prec r \prec \tilde{r} \prec q, \andSep
t \leq \varphi(r) \leq \varphi(\tilde{r}) \leq \tilde{t}.
\]
\end{lma}
\begin{proof}
Choose $\varepsilon>0$ such that
\[
t<\tilde{t}-2\varepsilon.
\]
Apply \autoref{prp:cutDown} to obtain $e\in\OP(A)$ such that
\[
p\prec e \prec \tilde{p}, \andSep
\varphi(\tilde{p}) -\varepsilon \leq \varphi(e).
\]
Then $e\prec q$, which allows us to set $z:=q-\overline{e}\in\OP(A)$.
We have $q\leq \tilde{p}+z$ and therefore
\[
0
< \tilde{t}-\varphi(\tilde{p})
\leq \varphi(q) -\varphi(\tilde{p})
\leq \varphi(z).
\]
%Apply \autoref{prp:NWSImpliesDenseOP} 
By assumption, we obtain $z'\in\OP(A)$ such that
\[
z'\leq z, \andSep
\varphi(z') \in (\tilde{t}-\varphi(\tilde{p})-\varepsilon, \tilde{t}-\varphi(\tilde{p})].
\]
Set $f := \tilde{p}\vee z' \in \OP(A)$.
Then
\[
\varphi(f)
= \varphi(\tilde{p}\vee z')
\leq \varphi(\tilde{p}) + \varphi(z')
\leq \varphi(\tilde{p}) + \tilde{t} - \varphi(\tilde{p}) = \tilde{t}.
\]
On the other hand, we have $e+z'\leq f$ and therefore
\[
\varphi(f)
\geq \varphi(e)+\varphi(z')
\geq \varphi(\tilde{p}) -\varepsilon + \tilde{t}-\varphi(\tilde{p})-\varepsilon
= \tilde{t}-2\varepsilon
> t.
\]
We also have $p\prec \tilde{p}\leq f$.
Applying \autoref{prp:cutDown} twice, we find $\tilde{r}$ and then $r$ such that
\[
p \prec r \prec \tilde{r} \prec f, \andSep
t<\varphi(r).
\]
Then $r$ and $\tilde{r}$ have the desired properties.
\end{proof}

%==========================================================================================
\begin{lma}
\label{prp:constructionPath}
Let $A$ be a \ca{}, let $\varphi\colon A\to\CC$ be a positive functional.
Assume that $\varphi([0,p])$ is dense in $[0,\varphi(p)]$ for every $p\in\OP(A)$.
Then there exist a path $f\colon(-\infty,0]\to\OP(A)$ (in the sense of \autoref{dfn:path}) satisfying $f(-\|\varphi\|)=0$ and
\[
\varphi(f(t))=\|\varphi\|+t
\]
for $t\in[-\|\varphi\|,0]$.
By \autoref{prp:elementFromPath}, $f$ corresponds to a positive element $a\in A$ with spectrum $[0,\|\varphi\|]$ on which $\varphi$ induces the Lebesgue measure.
\end{lma}
\begin{proof}
We may assume that $\varphi$ is nonzero, and by rescaling we may also assume that $\|\varphi\|=1$.
Using \autoref{prp:interpolatingPairs}, we inductively find $p_{k}^{(n)},\tilde{p}_{k}^{(n)}\in\OP(A)$ for $n\geq 1$ and $k=1,\ldots,2^n$ such that
\[
p^{(n)}_1 \prec \tilde{p}^{(n)}_1 \prec p^{(n)}_{2} \prec \tilde{p}^{(n)}_2 \prec \ldots \prec p^{(n)}_{2^n} \prec \tilde{p}^{(n)}_{2^n} \prec 1
\]
for each $n\geq 1$, and such that
\[
p^{(n)}_k 
\prec p^{(n+1)}_{2k} \prec \tilde{p}^{(n+1)}_{2k} 
\prec p^{(n+1)}_{2k+1} \prec \tilde{p}^{(n+1)}_{2k+1} 
\prec p^{(n)}_{k+1}
\]
and
\[
\tfrac{k-1}{2^n} 
< \varphi(p^{(n)}_k)
\leq \varphi(\tilde{p}^{(n)}_k)
\leq \tfrac{k-1}{2^n} + \tfrac{1}{2^{n+1}}
\]
for each $n\geq 1$ and $k=1,\ldots,2^n$.

Given $n\geq 1$ and $k\in\{1,\ldots,2^n\}$, we have 
\[
p^{(n)}_k \prec p^{(n+1)}_{2k} \prec p^{(n+2)}_{2^2k} 
\prec ... \prec p^{(n+m)}_{2^mk} \prec \ldots 
\]
and we set
\[
g(\tfrac{k}{2^n}) := \sup_m p^{(n+m)}_{2^mk} \in \OP(A).
\]
Note that this is well-defined, since if $\tfrac{k}{2^n}=\tfrac{k'}{2^{n'}}$ for some other $k',n'\geq 1$, then the sequence $(p^{(n'+m)}_{2^mk'})_m$ either contains $(p^{(n+m)}_{2^mk})_m$ as a subsequence (if $n'\leq n$) or vice verse (if $n\leq n'$).
Setting $g(0)=0$, we have defined $g(t)$ for every dyadic rational in $[0,1]$.

\emph{Claim~1: Let $n\geq 1$ and $k\in\{1,\ldots,2^n\}$.
Then $\varphi( g(\tfrac{k}{2^n}) ) = \tfrac{k}{2^n}$.}
Indeed, for each $m\geq 1$, we have
\[
\tfrac{k}{2^n} - \tfrac{1}{2^{n+m}}
= \tfrac{2^mk-1}{2^{n+m}} 
\leq \varphi\big( p^{(n+m)}_{2^mk} \big)
\leq \tfrac{2^mk-1}{2^{n+m}} + \tfrac{1}{2^{n+m+1}}
\leq \tfrac{k}{2^n},
\]
and therefore
\[
\varphi\big( g(\tfrac{k}{2^n}) \big)
= \varphi\big( \sup_m p^{(n+m)}_{2^mk} \big)
= \sup_m \varphi\big( p^{(n+m)}_{2^mk} \big)
= \tfrac{k}{2^n}.
\]

\emph{Claim~2: Let $t',t\in[0,1]$ be dyadic rationals satisfying $t'<t$.
Then $g(t')\prec g(t)$.}
To prove the claim, choose $n\geq 1$ and $k',k\in\{0,\ldots,2^n\}$ such that $t'=\tfrac{k'}{2^n}$ and $t=\tfrac{k}{2^n}$.
For each $m\geq 1$, we have
\[
p^{(n+m)}_{2^mk'} \prec p^{(n+m-1)}_{2^{m-1}k'+1} \prec p^{(n+m-2)}_{2^{m-2}k'+1} \prec \ldots \prec 
p^{(n+1)}_{2k'+1} \prec p^{(n)}_{k'+1}
\]
and therefore
\[
g(t')
= g(\tfrac{k'}{2^n})
= \sup_m p^{(n+m)}_{2^mk'}
\leq p^{(n)}_{k'+1}
\leq p^{(n)}_{k}
\prec p^{(n+1)}_{2k}
\leq \sup_m p^{(n+m)}_{2^mk}
= g(\tfrac{k}{2^n})
= g(t),
\]
which proves the claim.

We now define $f\colon(-\infty,0]\to\OP(A)$ by $f(t)=0$ for $t\leq -1$ and by
\[
f(t) := \sup \big\{ g(\tfrac{k}{2^n}) : \tfrac{k}{2^n}<1+t \big\}
\]
for $t\in(-1,0]$.
It is straightforward to verify that $f$ is order-preserving, and that $f(t)=\sup\{f(t'):t'<t\}$ for every $t$.

To verify condition~(2) of \autoref{dfn:path}, let $t'<t\leq 0$.
If $t'\leq -1$, then $f(t')=0\prec f(t)$.
Otherwise, choose dyadic rationals $s',s\in(0,1]$ such that
\[
1+t'\leq s'<s < 1+t.
\]
Using Claim~2, we get
\[
f(t')\leq g(s') \prec g(s) \leq f(t).
\]

Finally, let $t\in(-1,0]$.
Choose an increasing sequence $(s_k)_k$ of dyadic numbers in $[0,1+t)$ with supremum $1+t$.
Then $f(t)=\sup_k g(s_k)$. 
Using Claim~1, we get
\[
\varphi(f(t))
= \varphi\big( \sup_k g(s_k) \big) 
= \sup_k \varphi( g(s_k) ) 
= \sup_k s_k
= 1+t.
\]

By \autoref{prp:elementFromPath}, there is a unique positive element $a\in A$ such that $f(t)=\supp((a+t)_+)$ for each $t\leq 0$.
Let $\sigma(a)$ denote the spectrum of $a$, and let $\mu$ be the measure on $\sigma(a)$ induced by $\varphi$.
For every $t\geq 0$, we have
\[
\mu((t,\infty)\cap\sigma(a))
= \varphi( \supp( (a-t)_+ ) )
= \varphi( f(-t) ) 
= 1-t,
\]
which implies that $\sigma(a)=[0,1]$ and that $\mu$ is the Lebesgue measure on $\sigma(a)$.
\end{proof}

%==========================================================================================
Combining Lemmas~\ref{prp:NWSImpliesDenseOP} and~\ref{prp:constructionPath}, we obtain:

%==========================================================================================
\begin{prp}
\label{prp:NWSGivesPosElement}
Let $A$ be a \ca{}, and let $\varphi\colon A\to\CC$ be a nowhere scattered, positive functional.
Then there exists a positive element $a\in A$ with spectrum $[0,\|\varphi\|]$ on which $\varphi$ induces Lebesgue measure. % on $\sigma(a)$.
\end{prp}

%==========================================================================================
It will be convenient to generalize the notion of Haar unitaries to the setting of positive functionals that are not necessarily states or tracial.

%==========================================================================================
\begin{dfn}
\label{dfn:Haar}
Let $A$ be a unital \ca, and let $\varphi\colon A\to\CC$ be a positive functional.
A \emph{Haar unitary} in $A$ with respect to $\varphi$ is a unitary $u\in A$ such that $\varphi(u^k)=0$ for every $k\in\ZZ\setminus\{0\}$.
\end{dfn}

%==========================================================================================
\begin{prp}
\label{prp:charSingleHaar}
Let $A$ be a unital \ca{}, and let $\varphi\colon A\to\CC$ be a positive functional.
Then the following are equivalent:
\begin{enumerate}
\item
There exists a Haar unitary in $A$ with respect to $\varphi$.
\item
There exists a positive element $a\in A$ with spectrum $[0,\|\varphi\|]$ on which $\varphi$ induces the Lebesgue measure.
\item
There exists a unital, abelian sub-\ca{} $C(X)\subseteq A$ such that $\varphi$ induces a diffuse measure on $X$.
\item
There exists a maximal abelian subalgebra (masa) $D\subseteq A$ such that $\varphi$ induces a diffuse measure on the spectrum of $D$.
\end{enumerate}
\end{prp}
\begin{proof}
The statements hold for $\varphi=0$.
We may thus assume that $\varphi$ is nonzero.
To show that~(2) implies~(1), let $a\in A$ be as in~(2).
Set $u:=\exp(2\pi i \tfrac{a}{\|a\|})$, which is a unitary in $A$.
Given $k\in\ZZ\setminus\{0\}$, we have
\[
\varphi(u^k)
= \int_0^{\|a\|} \exp(2\pi i \tfrac{t}{\|a\|})^k dt
= \|a\| \int_0^{1} \exp(2\pi i t)^k dt
= 0.
\]

It is clear that~(4) implies~(3).
To show that~(3) implies~(2), let $C(X)\subseteq A$ be as in the statement.
Set $\psi:=\varphi|_{C(X)}\colon C(X)\to\CC$.
By assumption, $\psi$ is diffuse (see also \autoref{exa:diffuseMeasure}).
By \autoref{prp:diffuseImpliesNWS}, $\psi$ is nowhere scattered.
Since, $\|\psi\|=\|\varphi\|$, by applying \autoref{prp:NWSGivesPosElement}, we obtain the desired positive element in $C(X)$.
(This also follows from \cite[Proposition~4.1(i)]{DykHaaRor97SRFreeProd}.)

To show that~(1) implies~(4), let $u\in A$ be a Haar unitary.
Then $C^*(u)$, the sub-\ca{} of $A$ generated by $u$, is naturally isomorphic to $C(\TT)$, and $\varphi$ induces the multiple of the Lebesgue measure on $\TT$ with total mass $\|\varphi\|$.
Choose any masa $D\subseteq A$ that contains $u$, and let $X$ be a compact, Hausdorff space such that $D\cong C(X)$.
The inclusion $C(\TT)\cong C^*(u)\subseteq D\cong C(X)$ corresponds to a surjective, continuous map $X\to\TT$, and since $\varphi$ induces a diffuse measure on $\TT$, it also does on $X$.
(See also \autoref{prp:diffuseTrace-fullSubalg}.)
\end{proof}

%==========================================================================================
\begin{lma}
\label{prp:HaarImpliesNoScatteredQuotient}
Let $A$ be a unital \ca{}, let $\varphi\colon A\to\CC$ be a positive functional that admits a Haar unitary, and let $I\subseteq A$ be an ideal such that $A/I$ is scattered.
Then $\|\varphi|_I\|=\|\varphi\|$.
\end{lma}
\begin{proof}
%We may assume that $\varphi$ is nonzero.
Let $\pi\colon A\to A/I$ denote the quotient map. 
Choose a Haar unitary $u$ in $A$ with respect to $\varphi$, and let $B\subseteq A$ be the sub-\ca{} generated by $u$.
Then $B\cong C(\TT)$, and the measure $\mu$ induced by $\varphi$ on $\TT$ is a multiple of the Lebesgue measure.

The ideal $I\cap B$ of $B$ corresponds to a proper open subset $U\subseteq\TT$ such that the quotient $\pi(B)$ of $B$ is naturally isomorphic to $C(\TT\setminus U)$.
Since $A/I$ is scattered, so is $\pi(B)$, and it follows that $\TT\setminus U$ is countable.
Hence, $\mu(\TT\setminus U)=0$, which implies that $\mu(U)=\mu(\TT)$.
We deduce that
\[
\|\varphi|_I\|
\geq \|\varphi|_{I\cap B}\|
= \mu(U) 
= \mu(\TT)
= \|\varphi\|
\]
which implies the desired equality $\|\varphi|_I\|=\|\varphi\|$.
\end{proof}

%==========================================================================================
In the next result, given a hereditary subalgebra $B\subseteq A$ with corresponding open projection $p\in A^{**}$, we view the minimal unitization of $B$ as $\widetilde{B}=B+\CC p\subseteq A^{**}$.

%==========================================================================================
\begin{thm}
\label{prp:charNWS}
Let $A$ be a \ca{}, and let $\varphi\colon A\to\CC$ be a positive functional.
Then the following are equivalent:
\begin{enumerate}
\item
$\varphi$ is nowhere scattered.
\item
For every $p\in\OP(A)$, $\varphi([0,p])$ is dense in $[0,\varphi(p)]$.
\item
For every $p\in\OP(A)$ and $t\in[0,\varphi(p))$ there exists $p'\in\OP(A)$ with $p'\prec p$ and $\varphi(p')=t$.
\item
For every hereditary sub-\ca{} $B\subseteq A$ there exists a Haar unitary in~$\widetilde{B}$.
\item
For every ideal $I\subseteq A$ there exists a Haar unitary in $\widetilde{I}$.
\end{enumerate}
\end{thm}
\begin{proof}
By \autoref{prp:NWSImpliesDenseOP}, (1) implies~(2).
By \autoref{prp:constructionPath}, (2) implies~(3).
It is clear that~(3) implies~(2), and that~(4) implies~(5).

To show that~(5) implies~(1), let $I\subseteq J\subseteq A$ be ideals such that $J/I$ is scattered.
By assumption, there exists a Haar unitary in $\widetilde{J}$.
If $J$ is unital, then $\widetilde{J}=J$ and then $\widetilde{J}/I = J/I$.
If $J$ is nonunital, then $\widetilde{J}/I$ is naturally isomorphic to the forced unitization of $J/I$.
Thus, in either case, $\widetilde{J}/I$ is scattered.
Applying \autoref{prp:HaarImpliesNoScatteredQuotient} at the first step, we get
\[
\|\varphi|_I\|
= \|\varphi|_{\widetilde{J}}\|
= \|\varphi|_J\|.
\]

To show that~(2) implies~(4), let $B\subseteq A$ be a hereditary sub-\ca.
Consider the restriction $\psi := \varphi|_B$. % \colon B \to \CC$.
Then for every $p\in\OP(B)$, $\psi([0,p])$ is dense in $[0,\psi(p)]$.
Applying \autoref{prp:constructionPath}, we obtain $b\in B_+$ with spectrum $[0,\|\psi\|]$ on which $\psi$ induces Lebesgue measure.
Using that $\|\psi|_{\widetilde{B}}\|=\|\psi\|$, it follows from \autoref{prp:charSingleHaar} that~$\widetilde{B}$ contains a Haar unitary with respect to $\psi$, and hence with respect to $\varphi$.
\end{proof}

%==========================================================================================
\begin{cor}
\label{cor:noScatteredIdealQuotient}
Let $A$ be a unital \ca{} that has no (nonzero) scattered ideal-quotients.
Then every positive functional on $A$ is nowhere scattered and therefore admits a Haar unitary.
\end{cor}

%==========================================================================================
\begin{exa}
A \ca{} $A$ is \emph{purely infinite} if it has no one-dimensional representations, and if an element $a\in A_+$ lies in the ideal generated by $b\in A_+$ if and only if there exists a sequence $(r_n)_n$ in $A$ such that $\lim_{n\to\infty}\|a-r_nbr_n^*\|=0$;
see \cite[Definition~4.1]{KirRor00PureInf}.
By \cite[Theorem~4.19]{KirRor00PureInf}, pure infiniteness passes to ideals and quotients.
It follows that purely infinite \ca{s} are not scattered, and that a purely infinite \ca{} has no scattered ideal-quotients.

Hence, by \autoref{cor:noScatteredIdealQuotient}, every state on a unital, purely infinite \ca{} admits a Haar unitary.
\end{exa}

%==========================================================================================
We point out the following important special case of \autoref{cor:noScatteredIdealQuotient}.

%==========================================================================================
\begin{cor}
\label{prp:HaarSimple}
Let $A$ be a unital, simple, nonelementary \ca{}.
Then every positive functional $\varphi\colon A\to\CC$ admits a Haar unitary.
\end{cor}

%==========================================================================================
\begin{cor}
\label{prp:HaarSimpleSubalgebra}
Let $A$ be a unital \ca{}.
Assume that $A$ contains a unital sub-\ca{} $B\subseteq A$ such that $B$ has no scattered ideal-quotients.
(For example, $B$ is simple and nonelementary.)
Then every positive functional $\varphi\colon A\to\CC$ admits a Haar unitary.
\end{cor}

%==========================================================================================
\begin{exa}
\label{exa:rr0}
Let $A$ be a unital \ca{} of real rank zero that has no finite-dimensional representations.
By \cite[Corollary~2.4]{EllRor06Perturb}, $A$ contains a unital, simple, nonelementary sub-\ca{}.
Hence, by \autoref{prp:HaarSimpleSubalgebra}, every positive functional on $A$ admits a Haar unitary.
\end{exa}

%==========================================================================================
\begin{exa}
\label{exa:HaarNotImpliesNWS}
Let $(\delta_n)_{n\in\ZZ}$ be the standard orthonormal basis of $\ell^2(\ZZ)$, and let $U\in\Bdd(\ell^2(\ZZ))$ be the bilateral shift satisfying $U\delta_n=\delta_{n+1}$.
Let $\varphi\colon\Bdd(\ell^2(\ZZ))\to\CC$ be the vector state induced by $\delta_0$, that is, $\varphi(a)=\langle a\delta_0,\delta_0\rangle$ for $a\in\Bdd(\ell^2(\ZZ))$.
Then 
\[
\varphi(U^k)
= \langle U^k\delta_0,\delta_0\rangle
= \langle \delta_k,\delta_0\rangle
= \begin{cases}
0,& \text{ if } k\neq 0 \\
1,& \text{ if } k=0.
\end{cases}
\]
which shows that $U$ is a Haar unitary with respect to $\varphi$.
However, the ideal $I:=K(\ell^2(\ZZ))$ of compact operators is elementary and satisfies $\|\varphi|_I\|=\|\varphi\|$.
Thus, $\varphi$ is not nowhere scattered.

This shows that a positive functional admitting a Haar unitary is not necessarily nowhere scattered.
In the next section, we show that this phenomenon does not occur for tracial functionals.
%In particular, in condition~(4) of \autoref{prp:charNWS} it is really necessary to consider 
\end{exa}

%==========================================================================================
%==========================================================================================
\section{Traces admitting Haar unitaries}

%==========================================================================================
In this section, we prove the main result of this paper:
A tracial state on a unital \ca{} admits a Haar unitary if and only if it is diffuse;
see \autoref{prp:main}.
%We obtain that a unital \ca{} has no finite-dimensional representations if and only if every of its tracial states admits a Haar unitary;
%\autoref{prp:charNoFDRep}

%==========================================================================================
By a \emph{trace} on a \ca{} $A$ we mean a positive functional $\tau\colon A\to\CC$ that is tracial: $\tau(ab)=\tau(ba)$ for all $a,b\in A$.
Every trace $\tau\colon A\to\CC$ is bounded and therefore extends uniquely to a normal trace $A^{**}\to\CC$ that we also denote by $\tau$.
Recall that $\tau$ is diffuse if $\tau(e)=0$ for every minimal projection $e\in A^{**}$.

%==========================================================================================
\begin{lma}
\label{prp:charDiffuseTrace}
Let $A$ be a \ca{} and let $\tau\colon A\to\CC$ be a trace.
Then the following are equivalent:
\begin{enumerate}
\item
$\tau$ is diffuse;
\item
$\tau$ does not dominate a nonzero trace that factors through a finite-di\-men\-sion\-al quotient of $A$;
\item
there is no surjective ${}^*$-homomorphism $\pi\colon A\to M_n(\CC)$ (for some $n\geq 1$) such that $\tau$ dominates a nonzero multiple of $\tr_n\circ\pi$. where $\tr_n$ denotes the tracial state on $M_n(\CC)$.
\end{enumerate}
\end{lma}
\begin{proof}
To show that~(1) implies~(2), assume that $\tau$ is diffuse, and let $I\subseteq A$ be an ideal such that $A/I$ is finite-dimensional.
Let $z\in A^{**}$ be the central, open projection corresponding to $I$.
We have natural isomorphisms $(1-z)A^{**}\cong (A/I)^{**}\cong A/I$.
Since $\tau$ vanishes on minimal projections, we get $\tau(1-z)=0$.
It follows that $\tau$ does not dominate a nonzero trace that factors through the quotient map $A\to A/I$.

It is clear that~(2) implies~(3).
To show that~(3) implies~(1), let $e\in A^{**}$ be a minimal projection.
To reach a contradiction, assume that $\tau(e)>0$.
Let $c(e)\in A^{**}$ be the central cover of $e$.
Then $c(e)A^{**}$ is a type~$\mathrm{I}$~factor.
Since $\tau$ restricts to a nonzero trace on $c(e)A^{**}$, we have $c(e)A^{**}\cong M_n(\CC)$ for some $n$.
It follows that the map $\pi\colon A\to M_n(\CC)$, $a\mapsto c(e)a$, is surjective.
For each $a\in A_+$, we have $a\geq c(e)a$ in $A^{**}$ and therefore
\[
\tau(a)
\geq \tau(c(e)a)
= \tau(c(e))\cdot (\tr_n\circ\pi)(a).
\]
Since $\tau(c(e))\geq\tau(e)>0$, we have shown that $\tau$ dominates a nonzero multiple of $\tr_n\circ\pi$, contradicting the assumption~(3).
\end{proof}

%==========================================================================================
\begin{lma}
\label{prp:diffuseTrace-fullSubalg}
Let $A$ be a \ca{}, let $\tau\colon A\to\CC$ be a trace, and let $B\subseteq A$ be a sub-\ca{} such that $\tau|_{B}$ is diffuse and $\|\tau|_B\|=\|\tau\|$.
Then $\tau$ is diffuse.
\end{lma}
\begin{proof}
The inclusion map $B\to A$ induces a natural injective ${}^*$-homomorphism $B^{**}\to A^{**}$.
We let $p\in A^{**}$ denote the projection that is the image of the unit of~$B^{**}$.
Note that $p$ is the weak*-limit of any positive, increasing approximate unit of $B$ in $A^{**}$.
It follows that
\[
\tau(p)
= \|\tau|_B\|
= \|\tau\|
= \tau(1),
\]
and thus $\tau(1-p)=0$.

To show that $\tau$ is diffuse, let $e\in A^{**}$ be a minimal projection, and let $c(e)$ be its central cover.
To reach a contradiction, assume that $\tau(e)>0$.
Then $c(e)A^{**}$ is a type~$\mathrm{I}$~factor with a nonzero trace, and so $c(e)A^{**}\cong M_n(\CC)$ for some $n$.
We let $\pi\colon A\to M_n(\CC)$ denote the surjective ${}^*$-homomorphism $a\mapsto c(e)a$.

We distinguish two cases.
If $\pi(B)=\{0\}$, then $c(e)p=0$, and it follows that
\[
e\leq c(e)\leq 1-p
\]
and therefore $\tau(e)=0$.

If $\pi(B)$ is nonzero, then $\pi(B)$ is a nonzero finite-dimensional quotient of $B$.
Set $J := \ker(\pi)$, which is an ideal in $A$.
We naturally identify $J^{**} = (1-c(e))A^{**}$. 
%Let $z \in B^{**}$ denote the open projection corresponding to the ideal $J\cap B \subseteq B$.
Note that the projection $p \in A^{**}$ is the open projection corresponding to $BAB$, the hereditary sub-\ca{} generated by $B$.
It follows that $(1-c(e))p$ is the open projection in $A^{**}$ corresponding to $J\cap BAB$.
Using that $J \cap BAB = (J\cap B)A(J\cap B)$, we see that $(1-c(e))p$ belongs to $B^{**}$ -- it is the open projection in $B^{**}$ corresponding to the ideal $J\cap B \subseteq B$.

Hence, the nonzero projection $c(e)p$ belongs to $B^{**}$, and we have natural isomorphisms $c(e)pB^{**} \cong (B/(J\cap B))^{**}\cong \pi(B)^{**}\cong \pi(B)$.
Using that $\tau|_{B}$ is diffuse we deduce that $\tau(c(e)p)=0$.
It follows that $\tau(c(e))=0$ and so $\tau(e)=0$.
\end{proof}

%==========================================================================================
\begin{prp}
\label{prp:traceDiffuseNWS}
A trace on a \ca{} is diffuse if and only if it is nowhere scattered.
\end{prp}
\begin{proof}
Let $A$ be a \ca{}, and let $\tau\colon A\to\CC$ be a trace.
If $\tau$ is diffuse, then it is nowhere scattered by \autoref{prp:diffuseImpliesNWS}.
To show the converse, assume that $\tau$ is nowhere scattered.
By \autoref{prp:NWSGivesPosElement}, there exists $a\in A_+$ with spectrum $[0,\|\varphi\|]$, on which $\varphi$ induces Lebesgue measure.
Let $B\subseteq A$ be the sub-\ca{} generated by $a$.
Then $B$ is commutative and $\varphi|_B$ is diffuse (see also \autoref{exa:diffuseMeasure}).
We have $\|\varphi|_B\|=\|\varphi\|$, whence it follows from \autoref{prp:diffuseTrace-fullSubalg} that $\tau$ is diffuse.
\end{proof}

%==========================================================================================
We summarize our findings: % for the case of a tracial state on a unital \ca.
%An analogous result holds for traces on nonunital \ca{s}.

%==========================================================================================
\begin{thm}
\label{prp:main}
Let $\tau\colon A\to\CC$ be a trace on a \ca.
Then the following are equivalent:
\begin{enumerate}
\item
$\tau$ is diffuse;
\item
the weak*-closure of $A$ in the GNS-representation induced by $\tau$ is a diffuse von Neumann algebra;
\item
$\tau$ is nowhere scattered;
\item
$\tau$ does not dominate a nonzero trace that factors through a finite-di\-men\-sion\-al quotient of $A$;
\item
there exists $a\in A_+$ with spectrum $[0,\|\tau\|]$ on which $\tau$ induces the Lebesgue measure;
\item 
there exists a masa $C(X)\subseteq \widetilde{A}$ such that $\tau$ induces a diffuse measure on $X$;
\item
there exists a Haar unitary in $\widetilde{A}$.
\end{enumerate}
\end{thm}
\begin{proof}
By \autoref{prp:diffuseGNS}, (1) and~(2) are equivalent.
By \autoref{prp:traceDiffuseNWS}, (1) and~(3) are equivalent.
By \autoref{prp:charDiffuseTrace}, (1) and~(4) are equivalent.
By \autoref{prp:NWSGivesPosElement}, (3) implies~(5).
By \autoref{prp:charSingleHaar}, (5) implies~(6), and (6) is equivalent to~(7).

Lastly, let us show that~(6) implies~(1).
If $A$ is unital, this follows from \autoref{prp:diffuseTrace-fullSubalg} (see also \autoref{exa:diffuseMeasure}).
If $A$ is nonunital, consider the map $C(X) \to \widetilde{A} \to \widetilde{A}/A \cong \CC$, which corresponds to evaluation at some $x\in X$.
Note that $A\cap C(X)$ is naturally isomorphic to $C_0(X\setminus\{x\})$.
Since the measure on $X$ induced by $\tau$ is diffuse, it gives zero mass to $\{x\}$.
It follows that $\|\tau|_{A\cap C(X)}\| = \|\tau|_{C(X)}\| = \|\tau\|$ and that $\tau|_{A\cap C(X)}$ is diffuse.
Hence, $\tau$ is diffuse by \autoref{prp:diffuseTrace-fullSubalg}.
\end{proof}

%==========================================================================================
\begin{cor}
\label{prp:charNoFDRep}
A unital \ca{} has no finite-dimensional representations if and only if each of its tracial states admits a Haar unitary.
\end{cor}

%==========================================================================================
\begin{cor}
\label{prp:traceSimple}
Every trace on a unital, simple, nonelementary \ca{} is diffuse and admits a Haar unitary.
\end{cor}

%==========================================================================================
We end this section with an example of a diffuse trace on a unital \ca{} and a masa that contains no Haar unitary.

%==========================================================================================
\begin{exa}
\label{exa:masa}
Let $S\colon\ell^2(\NN)\to\ell^2(\NN)$ be the one-sided shift, and let $\mathcal{T}:=C^*(S)\subseteq\Bdd(\ell^2(\NN))$ be the Toeplitz algebra.
The compact operators $\mathcal{K}:=\mathcal{K}(\ell^2(\NN))$ are an ideal in $\mathcal{T}$ with $\mathcal{T}/\mathcal{K}\cong C(\TT)$, where $\TT\subseteq\CC$ denotes the unit circle.
We consider the masa $\ell^\infty(\NN)\subseteq\Bdd(\ell^2(\NN))$ and set $B:=\ell^\infty(\NN)\cap\mathcal{T}$.
Then $\mathcal{K}\cap B=c_0(\NN)\subseteq\ell^\infty(\NN)$.
Since the commutant of $c_0(\NN)$ in $\Bdd(\ell^2(\NN))$ is $\ell^\infty(\NN)$, we see that $B$ is a masa in $\mathcal{T}$.

We let $\pi\colon\mathcal{T}\to C(\TT)$ denote the quotient map.
Let $\tau\colon C(\TT)\to\CC$ be induced by the normalized Lebesgue measure on $\TT$.
Set $\varphi:=\tau\circ\pi$, which is a diffuse tracial state on $\mathcal{T}$.
Let $(e_n)_{n\in\NN}$ be the standard basis in $\ell^2(\NN)$.
We claim that
\[
\varphi(a) = \lim_{n\to\infty} \langle ae_n,e_n \rangle
\]
for each $a\in\mathcal{T}$.
Indeed, one can directly verify this for each $S^k(S^*)^l$ for $k,l\geq 0$, and since finite linear combinations of such elements are dense in $\mathcal{T}$, the formula holds for every $a\in\mathcal{T}$.
It follows that $B=c_0(\NN)+\CC 1$, and $\pi(B)\subseteq C(\TT)$ contains only the constant functions.

Thus, every unitary $u\in B$ satisfies $|\varphi(u)|=1$.
In particular, $B$ contains no Haar unitary.
To find a Haar unitary for $\varphi$, consider the function $v\colon\TT\to\TT$ satisfying $v(z)=z^2$ for $z$ with positive imaginary part, and satisfying $v(z)=z^{-2}$ for $z$ with negative imaginary part.
Since $v$ is of the form $v=\exp(ia)$ for a positive element $a\in C(\TT)$, we can lift $v$ to a unitary $u\in\mathcal{T}$ with $\pi(u)=v$.
Then $\varphi(u^k)=\tau(v^k)=\int_\TT v(z)^k dz = 0$ for $k\in\ZZ\setminus\{0\}$.
\end{exa}

%==========================================================================================
%==========================================================================================
\section{States admitting Haar unitaries}

%==========================================================================================
Let $\varphi\colon A\to\CC$ be a positive functional on a unital \ca.
In this section, we study when $\varphi$ admits a Haar unitary.
By \autoref{prp:HaarImpliesNoScatteredQuotient}, a necessary condition is that $\varphi$ gives no weight to scattered quotients of $A$.
We conjecture that this is also sufficient:

%==========================================================================================
\begin{cnj}
\label{cnj:Haar}
Let $A$ be a unital \ca, and let $\varphi\colon A\to\CC$ be a positive functional.
Then $\varphi$ admits a Haar unitary if and only if there is no ideal $I\subseteq A$ such that $A/I$ is scattered and $\|\varphi|_I\|<\|\varphi\|$.
\end{cnj}

%==========================================================================================
We confirm the conjecture in the case that $\varphi$ is tracial (\autoref{prp:main}), and if $A$ is a von Neumann algebra (\autoref{prp:charFctlVNHaar}).

Recall that a topological space $X$ is said to be $T_1$ if for every $x\in X$ the set $\{x\}$ is closed.

%==========================================================================================
\begin{lma}
\label{prp:PrimT1}
Let $\varphi\colon A\to\CC$ be a positive functional on a \ca{} $A$ whose primitive ideal space is $T_1$.
Then $\varphi$ is nowhere scattered if and only if $\varphi$ does not dominate a nonzero positive functional that factors through an elementary quotient.
\end{lma}
\begin{proof}
The forward implication is clear.
To show the converse, assume that $\varphi$ gives no weight to elementary quotients.
Using that the primitive ideal space is $T_1$, it follows that $\varphi$ gives no weight to elementary ideal-quotients.
By \autoref{prp:firstCharNWS}, this implies that $\varphi$ is nowhere scattered.
\end{proof}

%==========================================================================================
\begin{prp}
\label{prp:charHaarT1}
Let $A$ be a unital \ca{} whose primitive ideal space is $T_1$, and let $\varphi\colon A\to\CC$ be a positive functional.
Then the following are equivalent:
\begin{enumerate}
\item
$\varphi$ is nowhere scattered;
\item
$\varphi$ admits a Haar unitary;
\item
$\varphi$ does not dominate a nonzero positive functional that factors through a finite-dimensional quotient.
\end{enumerate}
If $A$ is also of type~$\mathrm{I}$, then these conditions are also equivalent to:
\begin{enumerate}
\setcounter{enumi}{3}
\item
$\varphi$ is diffuse.
\end{enumerate}
\end{prp}
\begin{proof}
By \autoref{prp:charNWS}, (1) implies~(2).
By \autoref{prp:HaarImpliesNoScatteredQuotient}, (2) implies~(3).
Since every unital, elementary \ca{} is a matrix algebra and therefore finite-di\-men\-sion\-al, it follows from \autoref{prp:PrimT1} that~(3) implies~(1).
If $A$ is also of type~$\mathrm{I}$, then the equivalence of~(1) and~(4) follows from \autoref{prp:typeI}.
\end{proof}

%==========================================================================================
\begin{exa}
\label{exa:liminal}
Recall that a \ca{} $A$ is \emph{liminal} (also called \emph{CCR}) if for every irreducible representation $\pi\colon A\to\Bdd(H)$ we have $\pi(A)=K(H)$.
Every liminal \ca{} is type~$\mathrm{I}$ and its primitive ideal space is $T_1$.
A unital \ca{} is liminal if and only if every of its irreducible representations is finite-dimensional.
\autoref{prp:charHaarT1} verifies \autoref{cnj:Haar} for liminal \ca{s}.

Recall that a \ca{} $A$ is \emph{subhomogeneous} if there exists $n\in\NN$ such that every irreducible representation of $A$ is at most $n$-dimensional.
Every subhomogeneous \ca{} is liminal.
We obtain in particular that a positive functional on a unital, subhomogeneous \ca{} admits a Haar unitary if and only if it does not dominate a nonzero positive functional that factors through a finite-dimensional quotient.
\end{exa}

%==========================================================================================
%==========================================================================================
\section{States on von Neumann algberas}

%==========================================================================================
In this section, we study when (normal) states on von Neumann algebras admit Haar unitaries.
We first show that a normal state is diffuse if and only if it is nowhere scattered.
We deduce that normal states on diffuse von Neumann algebras admit Haar unitaries.
However, it turns out that diffuseness is not necessary.
Indeed, the main result of this section, \autoref{prp:charVNHaar}, shows that every state on a von Neumann algebra without finite-dimensional representations admits a Haar unitary.
In particular, every state on $\Bdd(H)$ admits a Haar unitary;
see \autoref{rmk:stateBH}.

%==========================================================================================
\begin{lma}
\label{prp:vNdiffuse}
Let $M$ be a von Neumann algebra, and let $\varphi\colon M\to\CC$ be a \emph{normal}, positive functional.
Then the following are equivalent:
\begin{enumerate}
\item
$\varphi$ is diffuse (in the sense of \autoref{pgr:diffuse});
\item
$\varphi(e)=0$ for every minimal projection $e\in M$;
\item
$\varphi$ is nowhere scattered.
\end{enumerate}
\end{lma}
\begin{proof}
By \autoref{prp:diffuseImpliesNWS}, (1) implies~(3).
By \autoref{prp:firstCharNWS}, (3) implies~(2).
To show that~(2) implies~(1), assume that~$\varphi$ vanishes on every minimal projection in $M$, and let $e$ be a minimal projection in $M^{**}$.
Given a Banach space $E$, we use $\kappa_E\colon E\to E^{**}$ to denote the natural inclusion.
Let $M_*$ denote the predual of $M$ and set
\[
\pi := \kappa_{M_*}^* \colon M^{**}\cong (M_*)^{***} \to (M_*)^* \cong M.
\]
Then $\pi$ is a ${}^*$-homomorphism satisfying $\pi\circ\kappa_M=\mathrm{id}_M$.

Set $\bar{e}:=\pi(e)$.
Then $\bar{e}$ is a projection in $M$.
To see that it is minimal, let $x\in M$.
Since $e$ is minimal in $M^{**}$, there exists $\lambda\in\CC$ such that $exe=\lambda e$.
Then
\[
\bar{e}x\bar{e} 
= \pi(e)\pi(x)\pi(e)
= \pi(exe)
= \pi(\lambda e)
= \lambda \bar{e}.
\]

Thus, either $\bar{e}$ is zero, or $\bar{e}$ is a minimal projection in $M$.
In either case, we have $\varphi(\bar{e})=0$.
Using that $\varphi$ belongs to $M_*$, we get
\[
\varphi(e)
= \langle \kappa_{M_*}(\varphi), e \rangle_{M_*^{**},M^{**}}
= \langle \varphi, \kappa_{M_*}^*(e) \rangle_{M_*,M}
= \varphi(\bar{e})
= 0. \qedhere
\]
\end{proof}

%==========================================================================================
Recall that a von Neumann algbera is said to be \emph{diffuse} if it contains no minimal projections.

%==========================================================================================
\begin{prp}
\label{prp:charDiffuseVN}
A von Neumann algebra $M$ is diffuse if and only if every normal state on $M$ is diffuse.
\end{prp}
\begin{proof}
The forward implication follows from \autoref{prp:vNdiffuse}.
To show the converse, assume that $M$ is not diffuse.
Choose a minimal projection $e$ in $M$, and let $c(e)$ be its central cover.
Then $c(e)M$ is a type~$\mathrm{I}$~factor summand. 
Let $H$ be a Hilbert space such that $c(e)M\cong\Bdd(H)$, and choose a unit vector $\xi\in e(H)$.
The corresponding vector state $\varphi_0\colon\Bdd(H)\to\CC$, $a\mapsto\langle a\xi,\xi\rangle$ is normal and satisfies $\varphi_0(e)=1$.
Then $M\to\CC$, $a\mapsto\varphi_0(c(e)a)$, is a normal state that is not diffuse.
\end{proof}

%==========================================================================================
Let $M$ be a diffuse von Neumann algebra, and let $\varphi\colon M\to\CC$ be a normal state.
It follows from \autoref{prp:vNdiffuse} and \autoref{prp:charDiffuseVN} that $\varphi$ is nowhere scattered and therefore admits a Haar unitary by \autoref{prp:charNWS}.
This is well-known and follows for instance using that every maximal abelian sub-\ca{} (masa) $D\subseteq M$ is a diffuse, abelian von Neumann algebra, that $\varphi|_D$ is a normal trace, and that every normal trace on a diffuse, abelian von Neumann algebra admits a Haar unitary.

Thus, every masa $D\subseteq M$ contains a Haar unitary with respect to $\varphi$.
(For \ca{s}, this does not hold; see \autoref{exa:masa}.)
We note that this only holds for \emph{normal} states.
Indeed, given a masa $D\subseteq M$, we may choose a pure state on $D$ and extend it to a state $\psi$ on $M$.
Then $D$ contains no Haar unitary with respect to~$\psi$.
Nevertheless, in many cases we can find a different masa that contains a Haar unitary for $\psi$.
Indeed, in \autoref{prp:charVNHaar} we will show that this is always possible if~$M$ has no finite-dimensional representations.

\autoref{exa:HaarNotImpliesNWS} shows that a normal state on a von Neumann algebra may admit a Haar unitary without being diffuse.

%==========================================================================================
\begin{prp}
\label{prp:charFctlVNHaar}
Let $M$ be a von Neumann algebra, and let $\varphi\colon M\to\CC$ be a positive functional.
Then the following are equivalent:
\begin{enumerate}
\item
$\varphi$ admits a Haar unitary;
\item
$\varphi$ does not dominate a nonzero positive functional that factors through a finite-dimensional quotient of $M$.
\end{enumerate}
\end{prp}
\begin{proof}
By \autoref{prp:HaarImpliesNoScatteredQuotient}, (1) implies~(2).
To show the converse, assume that $\varphi$ gives no weight to finite-dimensional quotients of $M$.
For each $n\geq 1$, let $z_n$ be the central projection in $M$ such that $z_nM$ is the type~$\mathrm{I}_n$ summand of $M$.
It follows from the assumption that the restriction of $\varphi$ to $z_nM$ gives no weight to finite-dimensional quotients.
Since every irreducible representation of $z_nM$ is $n$-dimensional, we can apply \autoref{prp:charHaarT1} (see also \autoref{exa:liminal}) to obtain a Haar unitary $u_n\in z_nM$.

Set $z_{<\infty}:=\sum_{n=1}^\infty z_n$ and $z_{\infty}:=1-\sum_{n=1}^\infty z_n$.
Set $u:=\sum_{n=1}^\infty u_n \in z_{<\infty}M$.
Under the identification of $z_{<\infty}M$ with $\prod_{n=1}^\infty z_nM$, the unitary $u$ corresponds to $(u_n)_{n=1}^\infty$.
It follows that $u$ is a Haar unitary in $z_{<\infty}M$.

If $z_{\infty}=0$, then $u$ is the desired Haar unitary. 
So assume that $z_{\infty}\neq 0$.
Then $z_{\infty}M$ admits no finite-dimensional representations.
It follows that $z_{\infty}M$ contains a unital, simple, nonelementary sub-\ca{}, for example the hyperfinite $\mathrm{II}_1$~factor~$\mathcal{R}$ (we sketch the argument below).
Hence, the restriction of $\varphi$ to $z_{\infty}M$ admits a Haar unitary $v$;
see \autoref{prp:HaarSimpleSubalgebra}.
(Since von Neumann algebras have real rank zero, this also follows from \autoref{exa:rr0}.)
Now $u+v$ is the desired Haar unitary.

To complete the argument, let us show that a von Neumann algebra $N$ admits a unital embedding of $\mathcal{R}$ if and only if $N$ admits no finite-dimensional representations.
The forward implication is clear.
For the backward implication, we can use type decomposition to reduce to the cases that $N$ is properly infinite or type~$\mathrm{II}_1$.

If $N$ is properly infinite, then it follows from Propositions~V.1.22 and~V.1.36 in \cite{Tak02ThyOpAlgs1} that $N \cong N \bar{\otimes} \Bdd(\ell^2(\NN))$.
Using that $\mathcal{R}$ unitally embeds into $\Bdd(\ell^2(\NN))$, we get a unital embedding $\mathcal{R} \subseteq N$.

If $N$ is type~$\mathrm{II}_1$, then one can apply \cite[Proposition~1.35]{Tak02ThyOpAlgs1} to construct an increasing sequence $(A_n)_n$ of unital subalgebras of $N$ with $A_n \cong M_{2^n}(\CC)$.
Consider $N_0 := (\bigcup_n A_n)'' \subseteq N$.
We use the center valued trace $T\colon N \to Z(N)$;
see \cite[p.312ff]{Tak02ThyOpAlgs1} for details.
Since $T$ is normal and maps each $A_n$ into scalar multiples of the unit, it follows that $T$ maps $N_0$ into $\CC\subseteq N$ as well.
Thus, $N_0$ admits a faithful, normal, tracial state, and we see that this is also the unique normal tracial state on~$N_0$ using that each $A_n$ has a unique tracial state and that $\bigcup_n A_n$ is $\sigma$-weakly dense in $N_0$.
This shows that $N_0$ is a $\mathrm{II}_1$ factor.
Since $N_0$ is AFD, we have $N_0 \cong \mathcal{R}$ by Murray-von Neumann's uniqueness of the separable, AFD $\mathrm{II}_1$ factor $\mathcal{R}$.
\end{proof}

%==========================================================================================
\begin{thm}
\label{prp:charVNHaar}
Let $M$ be a von Neumann algebra.
Then the following are equivalent:
\begin{enumerate}
\item
$M$ has no finite-dimensional representations;
\item
every state on $M$ admits a Haar unitary;
\item
every tracial state on $M$ admits a Haar unitary.
\end{enumerate}
\end{thm}
\begin{proof}
It is clear that~(2) implies~(3), and that (3) implies~(1).
By \autoref{prp:charFctlVNHaar}, (1) implies~(2).
\end{proof}

%==========================================================================================
\begin{rmk}
\label{rmk:stateBH}
Let $H$ be a separable, infinite-dimensional Hilbert space.
By \autoref{prp:charVNHaar}, every state on $\Bdd(H)$ admits a Haar unitary and consequently restricts to a diffuse state on some masa.
This should be contrasted with the result of Akemann and Weaver, \cite{AkeWea08PureStateNotMultAnyMasa}, that the continuum hypothesis implies the existence of a pure state on $\Bdd(H)$ that does not restrict to a pure state on any masa.
\end{rmk}

%==========================================================================================
%==========================================================================================
\section{Traces on reduced group C*-algebras}
\label{sec:groups}

%==========================================================================================
Let $G$ be a discrete group, and let $\ell^2(G)$ be the associated Hilbert space with canonical orthonormal basis $(\delta_g)_{g\in G}$.
The \emph{left-regular representation} $\lambda_G$ is the representation of $G$ on $\ell^2(G)$ that maps $g\in G$ to the unitary $u_g\in\Bdd(\ell^2(G))$ satisfying $u_g\delta_h:=\delta_{gh}$ for $h\in G$.
The sub-\ca{} of $\Bdd(\ell^2(G))$ generated by $\{u_g:g\in G\}$ is called the \emph{reduced group \ca{}} of $G$, denoted by $C^*_\red(G)$. 
The vector $\delta_1\in\ell^2(G)$ induces a canonical tracial state $\tau_G\colon C^*_\red(G)\to\CC$ given by
\[
\tau_G(a) := \langle a\delta_1,\delta_1 \rangle
\]
for $a\in C^*_\red(G)$.
We have $\tau_G(u_1)=1$ and $\tau_G(u_g)=0$ for $g\in G\setminus\{1\}$.
It follows that $u_g$ is a Haar unitary with respect to $\tau_G$ if and only if $g$ has infinite order in $G$.

The von Neumann algebra generated by $\{u_g:g\in G\}$ is called the \emph{group von Neumann algebra} of $G$, denoted $L(G)$.

%==========================================================================================
\begin{prp}
\label{prp:charGpInfinite}
Let $G$ be a discrete group.
Then the following are equivalent:
\begin{enumerate}
\item
$G$ is infinite;
\item
the trace $\tau_G\colon C^*_\red(G)\to\CC$ is diffuse;
\item
the trace $\tau_G\colon C^*_\red(G)\to\CC$ admits a Haar unitary;
\item
$L(G)$ is diffuse.
\end{enumerate}
\end{prp}
\begin{proof}
By \cite[Proposition~5.1]{Dyk93FreeProdHyperfinite}, (1) implies~(4).
Conversely, if $G$ is finite, then $L(G)$ is finite-dimensional and therefore not diffuse.
By \autoref{prp:main}, (2) and~(3) are equivalent.
Note that $L(G)$ is the weak*-closure of $C^*_\red(G)$ under the GNS-representation induced by $\tau_G$.
Therefore, (2) and~(4) are equivalent by \autoref{prp:diffuseGNS}.
\end{proof}

%==========================================================================================
\begin{exa}
\label{exa:locFinGp}
Let $G$ be an infinite, discrete group.
Then $\tau_G$ admits a Haar unitary.
However, if $G$ is a torsion group (such as $G=\TT/\ZZ$), then none of the canonical unitaries $u_g$ ($g\in G$) is a Haar unitary since every element in $G$ has finite order.

If $G$ is locally finite, then even more is true:
There exists no Haar unitary for~$\tau_G$ in the group algebra $\CC[G]$.
Indeed, given $u\in\CC[G]$, since $u$ has finite support and since~$G$ is locally finite, there exists a finite subgroup $F\subseteq G$ such that $u$ belongs to $\CC[F]$.
But $\CC[F]$ is a finite-dimensional algebra and therefore does not contain Haar unitaries.
Thus, to find a Haar unitary for $\tau_G$, one really needs to go to the completion $C^*_\red(G)$ of $\CC[G]$. 
\end{exa}

%==========================================================================================
\begin{prp}
\label{prp:grpNoFD}
Let $G$ be a discrete group.
Then the following are equivalent:
\begin{enumerate}
\item
$G$ is nonamenable;
\item
$C^*_\red(G)$ has no finite-dimensional representations;
\item
every trace on $C^*_\red(G)$ admits a Haar unitary.
\end{enumerate}
\end{prp}
\begin{proof}
The equivalence between~(1) and~(2) is well-known.
For the convenience of the reader, let us sketch the proof.
In one direction, if $G$ is amenable, then the trivial representation is weakly contained in $\lambda_G$, which induces a one-dimensional representation of $C^*_\red(G)$.
Conversely, if $C^*_\red(G) \to M_n(\CC)$ is a (unital) representations, then the composition with the unique tracial state on $M_n(\CC)$ is an amenable trace on $C^*_\red(G)$, which implies that $G$ is amenable by \cite[Proposition~6.3.2]{BroOza08Book}.

By \autoref{prp:charNoFDRep}, (2) and~(3) are equivalent.
\end{proof}

%==========================================================================================
\begin{pgr}
Let $G$ be a discrete group.
Consider the following properties:
\begin{enumerate}
\item
$G$ contains a subgroup isomorphic to $\mathbb{F}_2$, the free group on two generators;
\item
$C^*_\red(G)$ contains a unital, simple, nonelementary sub-\ca{};
\item
every state on $C^*_\red(G)$ admits a Haar unitary;
\item
$C^*_\red(G)$ has no finite-dimensional representations;
\item
$G$ is nonamenable.
\end{enumerate}
Then the following implications hold:
\[
(1)\Rightarrow(2)\Rightarrow(3)\Rightarrow(4)\Leftrightarrow(5).
\]
Indeed, (1) implies that the simple, nonelementary \ca{} $C^*_\red(\mathbb{F}_2)$ unitally embeds into $C^*_\red(G)$;
by \autoref{prp:HaarSimpleSubalgebra}, (2) implies~(3);
and by \autoref{prp:grpNoFD}, (3) implies~(4), which is equivalent to~(5).

Recall that $G$ is \emph{$C^*$-simple} if $C^*_\red(G)$ is simple.
Obviously, every $C^*$-simple group satisfies~(2).
By \cite{OlsOsi14CsimpleGpsNoFreeSub}, there exist $C^*$-simple groups that have no noncyclic, free subgroups.
Hence, the implication '(1)$\Rightarrow$(2)' cannot be reversed.
What about the other implications?
\autoref{cnj:Haar} predicts that (4) implies~(3).
Does~(3) imply~(2)?
\end{pgr}

%==========================================================================================
%==========================================================================================
\section{Structure of reduced free products}
\label{sec:freeProd}

%==========================================================================================
Let $A$ and $B$ be unital \ca{s} with faithful tracial states $\tau_A$ and $\tau_B$, respectively. 
The \emph{reduced free product} of $(A,\tau_A)$ and $(B,\tau_B)$ is the (unique) unital \ca{} $C$ with faithful tracial state $\tau_C$ and unital embeddings $A\subseteq C$ and $B\subseteq C$ such that $\tau_C$ restrict to the given traces on $A$ and $B$, such that $A$ and $B$ generated $C$ as a \ca{}, and such that $A$ and $B$ are \emph{free} with respect to $\tau_C$, that is, $\tau_C(c_1c_2\cdots c_n)=0$ whenever $\tau_C(c_j)=0$ for all $j$ and either $c_1,c_3,\ldots\in A$ and $c_2,c_4,\ldots\in B$, or vice versa;
see Lecture~7, and in particular Definition~7.10 in \cite{NicSpe06LecturesFreeProb} for details.

One can think of this construction as a generalization of the free product of groups:
Given discrete groups $G$ and $H$, the reduced free product of the reduced group \ca{s} $C^*_\red(G)$ and $C^*_\red(H)$ with respect to their canonical tracial states is naturally isomorphic to $C^*_\red(G\ast H)$.

It is a well-studied problem to determine when a reduced free product $C$ is simple or has stable rank one (that is, the invertible elements in $C$ are dense).
In \cite{Avi82FreeProd}, Avitzour introduced the condition, later named after him, that there are unitaries $u,v\in A$ and $w\in B$ satisfying
\[
\tau_A(u)=\tau_A(v)=\tau_A(uv)=0, \andSep
\tau_B(w)=0.
\]
By \cite[Proposition~3.1]{Avi82FreeProd}, Avitzour's condition implies that $C$ is simple and has a unique tracial state. 
By \cite[Theorem~3.8]{DykHaaRor97SRFreeProd}, Avitzour's condition also implies that $C$ has stable rank one.

It is clear that Avitzour's condition is satisfied if $\tau_A$ and $\tau_B$ admit Haar unitaries.
Thus, it follows from \autoref{prp:main} that the reduced free product of two \ca{s} with respect to diffuse (faithful) tracial states is a simple \ca{} of stable rank one and with unique tracial state.
Using a result of Dykema, \cite[Theorem~2]{Dyk99SimplSRFreeProd}, it even suffices that one trace is diffuse and the other algebra is nontrivial:

%==========================================================================================
\begin{thm}
\label{prp:freeProdDiffuse}
Let $A$ and $B$ be unital \ca{s} with faithful tracial states $\tau_A$ and $\tau_B$, respectively.
Assume that $\tau_A$ is diffuse and that $B\neq\CC$.
Then the reduced free product of $(A,\tau_A)$ and $(B,\tau_B)$ is simple, has stable rank one and a unique tracial state.
\end{thm}
\begin{proof}
By \autoref{prp:main}, $A$ contains a unital, commutative sub-\ca{} $C(X)$ such that $\tau_A$ induces a diffuse measure on $X$.
This verifies the assumptions of \cite[Theorem~2]{Dyk99SimplSRFreeProd}, which proves the result.
\end{proof}

%==========================================================================================
\begin{cor}
\label{prp:freeProdSimpleTracial}
Let $A$ and $B$ be unital, simple \ca{s} with tracial states $\tau_A$ and $\tau_B$, respectively.
Assume that $A\neq\CC$ and $B\neq\CC$.
Then the reduced free product of $(A,\tau_A)$ and $(B,\tau_B)$ is simple, has stable rank one and a unique tracial state.
\end{cor}
\begin{proof}
If $A$ is infinite-dimensional, then $\tau_A$ is diffuse by \autoref{prp:traceSimple} and the result follows from \autoref{prp:freeProdDiffuse}.
The same argument applies if $B$ is infinite-dimensional.
If both $A$ and $B$ are finite-dimensional, then $A\cong M_m(\CC)$ and $B\cong M_n(\CC)$ for some $m,n\geq 2$, and in this case one can directly verify that Avitzour's condition is satisfied;
see \cite[Proposition~4.1(iv)]{DykHaaRor97SRFreeProd}
\end{proof}

%==========================================================================================
\begin{cor}
\label{prp:freeProdSimpleSR1}
The class of unital, simple, stable rank one \ca{s} with unique tracial state is closed under reduced free products.
\end{cor}
\begin{proof}
Let $A$ and $B$ be two \ca{s} in the considered class.
If $A\cong\CC$, then the reduced free product is isomorphic to $B$, which belongs to the class, and similarly if $B\cong\CC$.
We may therefore assume that $A\neq\CC$ and $B\neq\CC$.
Now the result follows from \autoref{prp:freeProdSimpleTracial}.
\end{proof}

%==========================================================================================
\begin{pgr}
Let $A$ and $B$ be unital \ca{s} with faithful states $\varphi$ and $\psi$, respectively.
Let $(C,\gamma)$ be the reduced free product of $(A,\varphi)$ and $(B,\psi)$, which is defined analogously to the tracial setting.
One can show that $\gamma$ is tracial if and only if~$\varphi$ and~$\psi$ are.

The \emph{centralizer} of $A$ with respect to $\varphi$ is defined as
\[
A_\varphi := \big\{ a\in A : \varphi(ab)=\varphi(ba) \text{ for all } b\in A \big\}.
\]
Note that $A_\varphi$ is a unital sub-\ca{} of $A$, and the restriction of $\varphi$ to $A_\varphi$ is tracial.
In this setting, Avitzour's condition is that there exist unitaries $u,v\in A_\varphi$ and $w\in B_\psi$ satisfying
\[
\varphi(u)=\varphi(v)=\varphi(uv)=0, \andSep
\psi(w)=0.
\]

Avitzour's condition still implies that the reduced free product is simple.
By \cite[Proposition~3.2]{Dyk99SimplSRFreeProd}, $C$ is also simple if $B\neq\CC$ and if there is a unital sub-\ca{} $C(X)\subseteq A_\varphi$ such that $\varphi$ induces a diffuse measure on $X$.
By \autoref{prp:main}, the condition on $A_\varphi$ is satisfied if and only if $\varphi|_{A_\varphi}$ is a diffuse trace.
In particular, we obtain that $C$ is simple if $\varphi$ is a diffuse trace on $A_\varphi$ and $B\neq\CC$.
\end{pgr}

%==========================================================================================
\begin{prp}
Let $A$ and $B$ be unital, simple, nonelementary \ca{s} with states $\varphi$ and $\psi$, respectively, and let $(C,\gamma)$ be their reduced free product. % of $(A,\varphi)$ and $(B,\psi)$.
If $\varphi$ and $\psi$ are tracial, then $C$ has stable rank one.
If $\varphi$ or $\psi$ is not tracial, then $C$ is properly infinite.
\end{prp}
\begin{proof}
The first statement follows from \autoref{prp:freeProdSimpleTracial}.
To show the second statement, assume that  $\varphi$ or $\psi$ is not tracial.
Then $\gamma$ is not tracial either.
By \autoref{prp:HaarSimple}, $\varphi$ and $\psi$ admit Haar unitaries.
They do not necessarily lie in the centralizers of $\varphi$ and $\psi$, but this is also not required to apply \cite[Theorem~4]{DykRor98ProjFreeProd}, which gives that $C$ is properly infinite.
\end{proof}

%==========================================================================================
\begin{pbm}
Describe when the centralizer of a state contains a Haar unitary. % (with respect to $\varphi$).
\end{pbm}

%==========================================================================================
\begin{rmk}
Let $M$ be a von Neumann algebra, and let $\varphi\colon M\to\CC$ be a faithful, normal state.
The centralizer $M_\varphi$ (defined as in the $C^*$-case) is a von Neumann subalgebra.
If $M$ is a factor of type~$\mathrm{II}$ or type~$\mathrm{III}_\lambda$ for $\lambda\in[0,1)$, then $M_\varphi$ is diffuse;
see \cite[Lemma~2.2]{MarUed14GeomVNPreduals}.
In this case, it follows that $M_\varphi$ contains a Haar unitary.

In general, this is not true.
Indeed, as shown in Section~3 of \cite{HerTak70StatesAutOA}, there exists a faithful, normal state $\varphi$ on the (unique) hyperfinite $\mathrm{III}_1$~factor such that $M_\varphi=\CC 1$;
see also \cite[Example~2.7]{MarUed14GeomVNPreduals}.
The construction actually shows that there exists a state on the CAR algebra with trivial centralizer.
\end{rmk}

%\bibliographystyle{../../aomalphaMyShort}
%\bibliography{../../References}

\providecommand{\bysame}{\leavevmode\hbox to3em{\hrulefill}\thinspace}
\providecommand{\noopsort}[1]{}
\providecommand{\mr}[1]{\href{http://www.ams.org/mathscinet-getitem?mr=#1}{MR~#1}}
\providecommand{\zbl}[1]{\href{http://www.zentralblatt-math.org/zmath/en/search/?q=an:#1}{Zbl~#1}}
\providecommand{\jfm}[1]{\href{http://www.emis.de/cgi-bin/JFM-item?#1}{JFM~#1}}
\providecommand{\arxiv}[1]{\href{http://www.arxiv.org/abs/#1}{arXiv~#1}}
\providecommand{\doi}[1]{\url{http://dx.doi.org/#1}}
\providecommand{\MR}{\relax\ifhmode\unskip\space\fi MR }
% \MRhref is called by the amsart/book/proc definition of \MR.
\providecommand{\MRhref}[2]{%
  \href{http://www.ams.org/mathscinet-getitem?mr=#1}{#2}
}
\providecommand{\href}[2]{#2}

\end{document}